\documentclass{amsart}

\usepackage[utf8]{inputenc}
\usepackage[T1]{fontenc}
\usepackage{graphicx}
\usepackage{subfigure}
\usepackage{epsfig}
\usepackage{geometry}
\usepackage{float}
\usepackage{empheq,etoolbox}

\DeclareMathOperator{\artanh}{artanh}

\usepackage{amsthm}

\usepackage{amsmath,amssymb,exscale}

\usepackage{framed}

\usepackage{enumerate}

\newtheorem{teo}{Theorem}[section]

\newtheorem{defi}[teo]{Definition}
\newtheorem{obs}[teo]{Remark}

\newtheorem{prop}[teo]{Proposition}
\newtheorem{coro}[teo]{Corollary}
\newtheorem{lema}[teo]{Lemma}

\title{Path Connectivity of Anosov Metrics on Surfaces}
\author{Guilherme Brandão Guglielmo \and  R. Ruggiero} 
\email{guilhermebrandaoguglielmo@gmail.com \and rafael.o.ruggiero@puc-rio.br}
\date{}

\thanks{This work was supported by CAPES and by the MathAmSud Project No. 88881.878892/2023-01.} 

\begin{document}
\begin{abstract}
We construct a class of Riemannian metrics in closed surfaces of genus greater than one, having Anosov geodesic flows, and some regions of positive curvature, such that for each such surface, there exists a smooth curve of conformal deformations that preserves the Anosov property and connects the surface with a Riemannian metric of negative curvature. The conformal deformation does not arise from geometric flows like the Ricci flow, since it is known that such flows might generate conjugate points in the presence of points of positive curvature in the surface.
\end{abstract}

\maketitle

\setcounter{secnumdepth}{-1}

\section{Introduction}

The theory of Anosov geodesic flows has an important role in the study of stable dynamics, with rich interactions between geometry, topology, and ergodic theory. The study of geodesic dynamics on compact surfaces of genus $\geq 0$ started with the works of Morse \cite{morse1924fundamental}, Hopf \cite{hopf1939ergodentheorie,hopf1948closed}, Hedlund \cite{hedlund1939dynamics}, and Green \cite{green1958theorem}. The hyperbolic dynamics of the geodesic flow of compact surfaces of negative curvature, explored by Hopf and Hedlund, was later extended by Anosov \cite{anosov1969geodesic} in his seminal work in the 1960's.  Pesin \cite{pesin1977geodesic,pesin1977characteristic,brin1974partially} and Eberlein \cite{eberlein1973geodesic,eberlein1973geodesic2,eberlein1973visibility} in the 1970's developed Anosov ideas in two different branches.  Pesin introduced the main tools (so far) to study the ergodic theory of invariant measures of nonuniformly hyperbolic dynamics, that may arise in the presence of regions of nonnegative curvature in the manifold. Eberlein developed a deep theory to understand the interaction between the global behavior of geodesics, the action of the fundamental group in the universal covering of the manifold, and the topological dynamics of the geodesic flow, of certain manifolds without conjugate points that may admit regions of positive curvature (visibility manifolds). Recall that a Riemannian manifolds has no conjugate points if the exponential map at each point is nonsingular. One of the remarkable results by Eberlein in this sense is the characterization of Anosov geodesic flows in compact manifolds without conjugate points assuming the linear independence of the so-called Green bundles (see \cite{eberlein1973geodesic} for the definition). In the case of Anosov geodesic flows,  Green bundles coincide with the dynamically defined invariant bundles of the dynamics. However, Green bundles always exist in compact manifolds without conjugate points, regardless of the sign of the sectional curvatures. 

Anosov showed that negative sectional curvatures ensure the Anosov property, and it is easy to construct examples of surfaces with nonpositive curvature and non Anosov geodesic flow. It is enough to have a whole closed geodesic along which the curvature vanishes. However, negative curvature is not necessary to grant the Anosov property. Gulliver in the 1970's \cite{gulliver1975variety} constructed compact surfaces with regions of positive curvature and Anosov geodesic flows. His construction was based on 'surgeries' starting from metrics with negative curvature. Such surfaces have the additional property of having no focal points, namely, spheres in the universal covering are strictly convex.  Manifolds without focal points are special examples of manifolds without conjugate points. Donnay and Pugh \cite{donnay2003anosov} further exhibited embedded high-genus surfaces in $\mathbb{R}^3$ with Anosov geodesic flow and many regions of positive curvature. K. Burns \cite{burns1992flat} gave an example of a compact surface without conjugate points having focal points, so the inclusion of manifolds without focal points in the family of manifolds without conjugate points is strict. 

Also in the 1970's, Klingenberg \cite{klingenberg1974riemannian} proved that if the geodesic flow of a compact manifold is Anosov, then the manifold has no conjugate points. By the $C^{1}$ structural stability of Anosov geodesic flows, this implies that the existence of a metric with Anosov geodesic flow in a compact manifold yields that the family of metrics without conjugate points is open in the $C^{2}$ topology. Klingenberg's result together with the results in the previous two paragraphs raise many interesting questions about the topological structure of the set of metrics without conjugate points in a compact manifold. How is the structure of metrics without focal points inside the metrics without conjugate points? It is clear that in an open $C^{2}$ neighborhood of metrics with negative curvature, the no focal points property  still holds. However, if we allow regions of zero curvature in a compact surface, we have already mentioned an example of a compact surface with nonpositive curvature (and hence without focal points) whose geodesic flow is not Anosov. Moreover, such surfaces are in the boundary of metrics without conjugate points in the $C^{2}$ topology. Is the set of metrics without conjugate points connected? What is the boundary of the set of metrics without conjugate points in the $C^{2}$ topology?. 

In 1991, Ruggiero \cite{ruggiero1991creation} showed that the $C^{2}$ interior of the set of metrics without conjugate points in a compact manifold is precisely the set of metrics whose geodesic flows are Anosov. This represented a first step toward the understanding of the topology of the set of metrics without conjugate points. However, the global topology of this set remains largely unexplored, with basic questions about connectedness and path-connectedness still unanswered, as well as the characterization of its boundary. 

 In the work by Jane and Ruggiero~\cite{jane2014boundary}, it was conjectured that the closure of the set of compact Riemannian surfaces with Anosov geodesic flows in the $C^{2}$ topology is the set of Riemannian metrics without conjugate points. The theory of the Ricci flow yields that higher genus compact surfaces with nonpositive curvature whose geodesic flows are not Anosov are indeed in the boundary of the set of metrics with Anosov geodesic flows (see \cite{jane2014boundary} for details). The main result by Jane-Ruggiero was to show that a certain metrics without focal points admiting regions of positive curvature, in the $C^{2}$ boundary of the set of metrics without conjugate points, are actually in the boundary of the set of Anosov metrics, namely, the set of Riemannian metrics whose geodesic flows are Anosov. In other words, it is shown that any metric $g_{0}$ in a  certain family of metrics without focal points in a compact surface can by accumulated by Anosov metrics in the $C^{2}$ topology. The key idea was to apply the theory of geometric flows, the so-called Ricci-Yang-Mills flow, to get a smooth curve $g_{t}$ where $t \in [0,\epsilon)$ of metrics starting at $g_{0}$ such that for every $t>0$, $g_{t}$ is an Anosov metric. 

The existence of this curve is not obvious at all. From the point of view of geometric flows like the Ricci flow for instance, it is not clear at all that such a path might exists just applying the Ricci flow. Actually, work in progress by Keith Burns, Solly Coles, Dong Chen and Florian Richter indicates that the Ricci flow on surfaces initially without conjugate points may produce metrics that do exhibit conjugate points at later times. Although these results are not yet fully established in the literature, they provide strong evidence that the absence of conjugate or focal points is not, in general, stable under Ricci flow.

The Ricci-Yang-Mills flow has a special feature that is not enjoyed by the Ricci flow: the so-called magnetic potential of the Ricci-Yang-Mills flow can be chosen in a way that the evolution of the metric by the flow reduces the regions of positive curvature of the surface and decreases the curvature pointwise in such regions. So we can control at least for small parameters of the flow the behavior of the evolution of the curvature under the flow. The evolution of the Gaussian curvature by the Ricci flow is a diffusion-reaction parabolic, partial differential equation, so the control of the behavior of the curvature seems much more difficult in this case (see \cite{jane2014boundary} for a detailed explanation). 

It looks natural to expect that the same idea would show the existence of smooth path of metrics inside the family of Anosov metrics and ending at a metric with strictly negative curvature. The problem here is that the study of the evolution of the Ricci-Yang-Mills flow has many technical issues inherent to the theory of parabolic partial differential equations. In particular, the study of longtime existence solutions of the flow is quite complicated and requires the application of a heavy analytic machinery.  

The main contribution of this article to the study of Anosov metrics combines Gulliver's idea to get examples of surfaces without focal points and Anosov geodesic flows, with Jane-Ruggiero idea of deforming surfaces controlling the behavior of regions of positive curvature until they desappear from the surface. We show that there exist a family of compact surfaces without focal points, whose regions of positive curvature are concentrated in a finite number of small disks, such that their geodesic flows are Anosov, and there exist smooth paths of Anosov metrics connecting them with metrics of negative curvature. The construction of the smooth paths does not involve geometric flows, it is indeed much more simple but not canonical.

To define this family of surfaces, let us first introduce some notations. Let $\mathcal{M}(\delta, k, \epsilon, \Lambda)$, consists of compact Riemannian surfaces of genus greater than one, such that 
\begin{enumerate}
\item There exist $k$ disjoint balls of radius $\delta$, such that their lifts in the universal covering are pairwise distant by at least $\Lambda$, containing the points of positive curvature, called \emph{generalized bubbles},
\item In the complement of these generalized bubbles the curvature is less than $-\epsilon$.
\end{enumerate}

The main results of this paper are the following:

\begin{teo}\label{Teorema 1 introdução}
Let $0<\epsilon<1$ and $\delta>0$. There exists $\Lambda=\Lambda(\epsilon,\delta)>0$ such that, if $(M,g)\in\mathcal{M}(\delta,k,\epsilon,\Lambda)$ satisfies
$$
K^{+}<\frac{\sqrt{\epsilon}\left(1-\frac{\epsilon}{2}\right)
-\epsilon(2\delta+1)\left(1-\frac{\epsilon}{2}\right)^2}{2\delta},
$$
then:
\begin{enumerate}
\item $(M,g)$ has no focal points;
\item the geodesic flow of $g$ is Anosov.
\end{enumerate}
\end{teo}

We would like to remark that when
$$
\delta<\frac{1}{2}\left(
\frac{1}{\sqrt{\epsilon}\left(1-\frac{\epsilon}{2}\right)}-1
\right),
$$
the upper bound for $K^{+}$ in Theorem \ref{Teorema 1 introdução} is positive. When it is nonpositive, Theorem \ref{Teorema 1 introdução} can be proven applying the Ricci flow as we mentioned before.

\begin{teo}\label{Teorema 3 introdução}
Let $0<\epsilon<1$ and $\delta>0$, and set
$$
\Lambda := \frac{1}{\sqrt{\epsilon}}\,
\mathrm{artanh}\!\left(1-\tfrac{\epsilon}{2}\right).
$$
Assume that $(M,g)\in\mathcal{M}(\delta,k,\epsilon,\Lambda)$ and that
$$
\epsilon < \frac{-2\pi\,\chi(M)}{\operatorname{vol}(M)}.
$$
Then there exists $w\in C^{\infty}(M)$ such that, if the maximal curvature $K^{+}$ of $(M,g)$ satisfies
$$
K^{+}
< \frac{\sqrt{\epsilon}}{4e^{2\mu}\delta}
\left[
\tanh\!\left(e^{-\mu}\tfrac{\ln 3}{3}\right)
-
\sqrt{\epsilon}\,e^{-\mu}\,
\tanh^{2}\!\left(e^{-\mu}\tfrac{\ln 3}{3}\right)
(4e^{\mu}\delta+1)
\right],
$$
with $\mu:=\max_{M}w$, the following hold:
\begin{enumerate}
\item the metrics $g_{\rho}=e^{2\rho w}g$, $\rho\in[0,1]$, are Anosov and have no focal points;
\item the metric $g_{1}$ has strictly negative sectional curvature.

\end{enumerate}
\end{teo}

As in Theorem \ref{Teorema 1 introdução}, when 
$$
\delta<\frac{1}{4e^{\mu}}\left(
\frac{1}{\sqrt{\epsilon}e^{-\mu}\tanh\left(e^{-\mu}\frac{\ln 3}{3}\right)}-1
\right),
$$
the upper bound for $K^{+}$ is positive.

Some remarks about the statements of the Theorems. The family of surfaces considered in Theorem \ref{Teorema 1 introdução} is not empty. Indeed, the distance $\Lambda$ between different lifts of generalized bubbles in the universal covering can be achieved taking surfaces with high genus and just one generalized bubble for instance. A high genus correspond to large area and diameter of fundamental domains of the surface if the curvature is constant and negative. The starting idea of Gulliver's examples is to make surgery on surfaces like the above ones, deforming the surfaces in small neighborhoods and gluing disks with positive curvature. The hypothesis about the upper bound of the negative curvature outside the generalized bubbles is a mild assumption on the curvature since in general, the negative curvature outside the bubbles might be much more negative due to Gauss-Bonnet Theorem. Moreover, the upper bound on $K^{+}$ in the statement of Theorem \ref{Teorema 1 introdução} is positive if $0< \epsilon <1$ and $\delta$ is small enough. Notice that $\delta$ is an upper bound for the radius of the generalized bubbles.

Analogously, the family of surfaces considered in Theorem \ref{Teorema 3 introdução} is non empty. However, to see that involves a more subtle argument. This family is contained in the family considered in Theorem \ref{Teorema 1 introdução}, with an additional assumption involving a smooth positive function $w$ that is not explicitly defined (see Section \ref{seção 4} for details). The upper bound for $K^{+}$ in the statement of Theorem \ref{Teorema 3 introdução} depends on $w$, and again we can verifiy that the curvature of surfaces which are close to surfaces of nonpositive curvature where the zero curvature points are contained in a finite set of disjoint, open neighborhoods, satisfy this bound. A more detailed discussion about the subject is made in Section \ref{seção 5}.

Let us briefly describe the contents of the article. 

In Section \ref{seção 1} we collect the necessary background and preliminary facts manifolds without conjugate points and without focal points. This includes brief review of the relationship between the behaviour of solutions of the well known Riccati equation associated to geodesics without conjugate points and the characterization of no focal points in terms of their geometry. We also show, by a standard argument, that every point admits a neighbourhood free of focal points.

In Section \ref{seção 2} we recall the definition of Anosov geodesic flows and review their relationship with solutions of the Riccati equation, Namely, following Eberlein's work, the geodesic flow of the surface is Anosov if and only if for each geodesic the maximal and the minimal solution of the corresponding Riccati equation, defined in the whole real line, are different. Moreover, the surface has no focal points if and only if each of these solutions is either everywhere nonpositive or everywhere nonnegative. 

The purpose of Section \ref{seção 3} is to prove Theorem \ref{Teorema 1 introdução}. The main idea of the section is to apply the comparison theorems of the Riccati equation to the family of surfaces considered in Theorem \ref{Teorema 1 introdução}, to ensure that the Riccati equation on each geodesic of the surface has two different solutions defined in the whole real line, satisfying the conditions described in the previous paragraph. 

In Section \ref{seção 4} we exhibit a explicit conformal deformation of metrics that does not rely on geometric flows arguments. The motivation for the definition of the conformal deformations is the main idea introduced by Jane-Ruggiero in \cite{jane2014boundary} to try to "approach" Anosov surfaces starting from a surface without focal points. Namely, the intuitive idea that reducing the regions of positive curvature and decresing pointwise the curvature in these regions should provide some hyperbolicity to the geodesic flow. The conformal deformations introduced in Section \ref{seção 4} are based on the classical formula of the conformal curvature, we show that for appropriate conformal deformations of the Riemannian metric, we succeed in controlling the behavior of the curvature along the conformal deformation following Jane-Ruggiero's idea. The conformal factor $w$ of the deformation, mentioned in Theorem \ref{Teorema 3 introdução}, is obtained as the solution of an inverse problem of the Laplacian of the surface, This is why the function $w$ is not explicit. The conformal deformations are not canonical, there are infinitely many choices. We just consider a natural one based on Gauss-Bonnet Theorem. 

Finally, in Section \ref{seção 5} we reuse the arguments of Section \ref{seção 3}, in a more technical form, to derive conditions on the surface under which the deformations introduced in Section \ref{seção 4} actually preserve the Anosov property and the absence of focal points. This is the most technical part of the article, and the final goal of eliminating the regions of positive curvature through the conformal deformations is attained, thus proving Theorem \ref{Teorema 3 introdução}. 

We think that the techniques applied in the article can be used to show that Jane-Ruggiero's examples of surfaces without focal points can be connected by a conformal path of metrics preserving the Anosov property to a metric of negative curvature. The difficulty here regarding the ideas applied in the present article is that the regions of positive curvature considered in Jane-Ruggiero's paper have no a priori restrictions, so the conformal deformation proposed in the present article might generate focal points. We shall deal with this problem in a further article. 

\setcounter{secnumdepth}{1}
\section{Preliminaries: Surfaces free of focal points and Riccati equation}\label{seção 1}

In this section, we present fundamental concepts about surfaces without focal points that form the basis of this study. The exposition is based primarily on  \cite{pesin1977geodesic} and \cite{eberlein1973geodesic}.

Throughout this paper, we consider $(M,g)$ to be a $C^{\infty}$ Riemannian manifold. We denote by $(\tilde{M},\tilde{g})$ its universal covering, where $\tilde{g}$ is the pullback of $g$ via the covering map. Many results will be formulated or interpreted in terms of $(\tilde{M},\tilde{g})$ when convenient.

We begin with some fundamental definitions from the theory of focal points, including the notions of open sets without focal points and Riemannian manifolds without focal points.

\begin{defi}
     Let $(M,g)$ be a $C^{\infty}$ Riemannian manifold and $\gamma:[0,l]\to M$ a geodesic. 
     \begin{enumerate}
         \item Let $\gamma:[0,l]\to M$ be a geodesic. Then $\gamma$ has no focal points in $[0,l]$ if and only if for every Jacobi field $J(t)$ along $\gamma$, perpendicular to $\gamma'(t)$ and satisfying $J(0)=0$, the norm of $J(t)$ is increasing for $t \in [0,l]$.

         \item Let $B$ be an open set in $(M,g)$. We say that $B$ is free of focal points if the geodesic segments contained in $B$ are free of focal points. 
         
         \item  We say that $(M,g)$ \textit{has no focal points} if every geodesic in $(M,g)$ is free of focal points.
     \end{enumerate}
\end{defi}

When studying Riemannian manifolds, Jacobi fields play a central role in understanding the behavior of geodesics. We now introduce some concepts related to Jacobi fields on manifolds without conjugate points or without focal points.

\begin{defi}
    Let $(M,g)$ be a compact Riemannian manifold without conjugate points. Let $\gamma$ be a geodesic parametrized by arc length in $M$. Suppose $V$ is a vector orthogonal to $\gamma'(0)$. The \textit{stable Jacobi field} $J^{s}_{V}$ with initial condition $J^{s}_{V}(0) = V$ is given by
    $$J^{s}_{V}(t)=\lim_{T\to +\infty}J_{T}(t)$$
    where $J_{T}$ is a Jacobi field along $\gamma$ with boundary conditions $J_{T}(0) = V$ and $J_{T}(T) = 0$. The \textit{unstable Jacobi field} $J^{u}_{V}$ with initial condition $J^{u}_{V}(0) = V$ is given by
    $$J^{u}_{V}(t)=\lim_{T\to -\infty}J_{T}(t).$$
\end{defi}

Jacobi fields as described above are known as Green's Jacobi fields. These fields are important role in the study of Anosov metrics.

Let $\theta = (p, u) \in T_1M$, and denote by $O_{\theta}$ the subspace of $T_{\theta}T_1M$ orthogonal to the geodesic vector field with respect to the Sasaki metric. Recall that, under the Sasaki metric, the tangent bundle $TT_1M$ can be understood as the direct sum of the horizontal and vertical subbundles. The called \textit{Green bundles} (or \textit{Green subspaces}) are subspaces of $O_{\theta}$ defined by
\begin{equation*}
E^s(\theta) = \bigcup_{V \in (\gamma'_{\theta})^{\perp}} \left\{ \left( J_V^s(0), {J_V^s}'(0) \right) \right\}, \quad
E^u(\theta) = \bigcup_{V \in (\gamma'_{\theta})^{\perp}} \left\{ \left( J_V^u(0), {J_V^u}'(0) \right) \right\},
\end{equation*}
where $J_V^s$ and $J_V^u$ denote the stable and unstable Green Jacobi fields, respectively. 

The subspace $E^s(\theta)$ is called the \textit{stable Green subspace}, while $E^u(\theta)$ is the \textit{unstable Green subspace}. The bundles $E^s$ and $E^u$ are invariant under the geodesic flow, but are not necessarily continuous in general. However, if the metric is Anosov or the manifold has no focal points, then these subbundles are indeed continuous.

Now suppose that $(M,g)$ is a compact Riemannian surface without conjugate points, and let $\gamma$ be a geodesic with $||\gamma'||=1$ in $M$. A Jacobi field along $\gamma$ perpendicular to $\gamma'$ can be written as
$$J(t)=f(t)e(t)$$
where $e(t)$ is a unit parallel vector field along $\gamma$ perpendicular to $\gamma'$ and $f(t)$ is a solution to the scalar Jacobi equation
$$f''(t)+K(t)f(t)=0.$$
We then define
$$U(t):=\frac{f'(t)}{f(t)}$$
which is a solution to the \textbf{Riccati equation}
$$U'(t)+U^{2}(t)+K(t)=0.$$

Again, solutions to the Riccati equation may have asymptotes since $f(t)$ may be zero for some $t \in \mathbb{R}$. However, if $J(t)$ is a Green's Jacobi field, the Riccati solution induced by it is entire. In particular, we have the following:
\begin{defi}
   Consider $\theta = (p,v)$ and $W \in T_{1}M$ be perpendicular to $v$. Define the \textit{stable solution of the Riccati equation} $U^{s}_{\theta}$ as the solution of the Riccati equation induced by $J^{s}_{W}$ and the \textit{unstable solution of the Riccati equation} $U^{u}_{\theta}$ as the Riccati solution induced by $J^{u}_{W}$.
\end{defi}
In particular, $U^s_{\theta}$ and $U^u_{\theta}$ are defined for all $t \in \mathbb{R}$. It is possible to study the absence of focal points through the Riccati equation. We consider the following properties:

\begin{lema}
    Let $(M,g)$ be a compact surface without focal points. Then
    \begin{enumerate}
        \item The norms of stable Jacobi fields are non-increasing. The norms of unstable Jacobi fields are non-decreasing.
        \item A Jacobi field is parallel if and only if it is both unstable and stable at the same time.
        \item Given $\theta \in T_{1}M$, we denote by $U^{s}_{\theta}$ and $U^{u}_{\theta}$ the stable and unstable solutions of the Riccati equation on $\gamma_{\theta}$. The other solutions of the Riccati equation tend to $U_{\theta}^{u}(t)$ as $t \to +\infty$ and to $U_{\theta}^{s}(t)$ as $t \to -\infty$.
        \item The solutions $U^{s}_{\theta}$ and $U^{u}_{\theta}$ depend continuously on $\theta$.
    \end{enumerate}
\end{lema}

In fact, we make use of the following equivalence in this work:

\begin{lema}\label{lema instável não negativo e estável não positivo}
    Let $(M,g)$ be a compact surface. Then $(M,g)$ has no focal points if and only if $U^u(t) \ge 0$ and $U^s(t) \le 0$ for all $t$ and for all geodesics $\gamma$.
\end{lema}

\begin{proof}
$\Rightarrow)$ Suppose that $(M,g)$ has no focal points. Let $\gamma$ be a geodesic and $J_T$ be a Jacobi field perpendicular to $\gamma$ such that $J_T(0)$ has norm $1$ and $J_T(T) = 0$. Since $(M,g)$ is a surface without focal points, we allow an abuse of notation and write $J_T'(t) > 0$ for all $t \ge T$.

But,
$$J^u(t) = \lim_{T \to -\infty} J_T(t), \quad (J^u)'(t) = \lim_{T \to -\infty} J_T'(t).$$
It follows that $(J^u)'(t) \ge 0$, and hence $U^u(t) \ge 0$ for all $t$. Similarly, we have $U^s(t) \le 0$.

$\Leftarrow)$ Now suppose that $U^u(t) \ge 0$ and $U^s(t) \le 0$ for all $t$ and all geodesics $\gamma$. The existence of $U^u$ and $U^s$ implies that there are no conjugate points. Let $J$ be a Jacobi field perpendicular to $\gamma$ such that $J(T) = 0$. Consider the Riccati solution associated to $J$ for $t > T$, that is,
$$U(t) = \frac{J'(t)}{J(t)} \quad \text{for all } t > T.$$
Clearly, $U$ has a vertical asymptote at $t = T$.

On the other hand, since the surface is compact (and thus the curvature is bounded), a standard comparison argument for Riccati equations, in the style of Sturm–Liouville and as done in \cite{eberlein1973geodesic}, gives
$$\lim_{t \to T^+} U(t) = +\infty.$$

Since the unstable solution $U^u$ is defined for all $t$, by uniqueness of solutions to the Riccati equation, $U^u$ is a lower bound for $U$, and therefore $U(t) > 0$ for all $t > T_0$. This ensures that $J(t)$ is increasing for $t \ge T_0$.

Hence, since this holds for any radial Jacobi field, $(M,g)$ has no focal points.
\end{proof}

At the beginning of this section, we introduced definitions related to focal points, including open sets without focal points. One might initially think that the absence of focal points is a rare or rigid property. However, this is not the case: around any point on a compact Riemannian surface, there exists a neighborhood free of focal points.

\begin{lema}\label{proposição da prova que existe vizinhanças sem pontos focais}
   Let $(M,g)$ be a compact Riemannian surface, and let $K^+ = \max K$ be the maximal sectional curvature on $M$. Then, for any point $p \in M$, the ball $B(p, \frac{\pi}{4\sqrt{K^+}})$ is free of focal points.

In other words, given $\delta > 0$, there exists $K^+ > 0$ such that if the sectional curvature of $(M,g)$ satisfies $K \le K^+$, then every ball of radius $\delta$ is free of focal points.
\end{lema}

\begin{proof}
Fix $p \in M$, and let $\gamma$ be a geodesic with $\gamma(0) = p$. Let $J$ be a Jacobi field along $\gamma$ such that $J(0) = 0$ and $J'(0) > 0$. Consider the comparison surface $S_{K^+}$ with constant curvature $K^+$, and the Jacobi field $J_{K^+}$ along the corresponding geodesic on $S_{K^+}$ with the same initial conditions.

The general solution of the Jacobi equation on $S_{K^+}$ is
$$
J_{K^+}(t) = B \sin(\sqrt{K^+}t), \quad J_{K^+}'(t) = \sqrt{K^+} B \cos(\sqrt{K^+}t),
$$
and $J_{K^+}$ is strictly increasing on $[0, \frac{\pi}{2\sqrt{K^+}})$. Thus, the Riccati solution $V(t) = \frac{J_{K^+}'(t)}{J_{K^+}(t)}$ satisfies $V(t) > 0$ on $(0, \frac{\pi}{2\sqrt{K^+}})$.

Now fix $0 < \delta < \frac{\pi}{2\sqrt{K^+}}$ and define a Jacobi field $J_\delta$ along $\gamma$ such that $J_\delta(\delta) = J_{K^+}(\delta)$ and $J_\delta'(\delta) = J_{K^+}'(\delta)$. Let $U_\delta$ be the associated Riccati solution. Then,
\begin{align}
    U_\delta'(t) + U_\delta^2(t) + K(t) &= 0, \label{eq:U_riccati1}\\
    V'(t) + V^2(t) + K^+ &= 0.\label{eq:V_riccati1}
\end{align}

Subtracting the equations \ref{eq:U_riccati1} and \ref{eq:V_riccati1} at $t = \delta$, we get
$$
U_\delta'(\delta) - V'(\delta) = K^+ - K(\delta) \ge 0,
$$
so $U_\delta(t) > V(t)$ near $t = \delta$. By continuity and the comparison principle, $U_\delta(t) > V(t)$ for all $t$ in $(0, \delta)$. Passing to the limit as $\delta \to 0$, we obtain that the Riccati solution $U$ associated to $J$ satisfies $U(t) \ge V(t) > 0$ on $(0, \frac{\pi}{2\sqrt{K^+}})$. Hence, $J(t)$ is strictly increasing on this interval.

This shows that if $p_1 \in B(p, \frac{\pi}{2\sqrt{K^+}})$, then $p$ and $p_1$ are not focal points of each other. In particular, if $p_1, p_2 \in B(p, \frac{\pi}{4\sqrt{K^+}})$, then the distance between them is less than $\frac{\pi}{2\sqrt{K^+}}$, and they are not focal points of each other.

To conclude the second statement, given $\delta > 0$, just choose $K^+$ such that $\delta \le \frac{\pi}{4\sqrt{K^+}}$, or equivalently, $K^+ \le \frac{\pi^2}{16\delta^2}$. Then, by the same argument, each ball of radius $\delta$ is free of focal points.
\end{proof}

\section{Preliminaries: Compact Surfaces, Anosov Metrics and Riccati Equation} \label{seção 2}

In this section, we collect some supplementary results on compact surfaces and Riccati equation solutions, as well as remarks on Anosov metrics.

\begin{prop}[\cite{eberlein1973geodesic}, Proposition 2.7]\label{lema da cota fora da vizinhança da assintota}
If $(M,g)$ is a compact surface, then for each $a>0$ there exists $C=C(K,a)>0$ such that any Riccati solution $U$ with an asymptote at $T$ we have that
$$
\|U(t)\| < C
\quad
\text{for all }t\notin (T-a,\,T+a).
$$
\end{prop}

Then, follows from the Proposition \ref{lema da cota fora da vizinhança da assintota},

\begin{lema}\label{lema da convergência da solução da equação de riccati não negativa}
Let $(M,g)$ be a compact surface and let $\gamma$ be a geodesic in $(M,g)$.

\begin{itemize}
  \item If there is a sequence of Riccati solutions $\{U_n\}$ along $\gamma$ each having an asymptote at $t_n$, satisfying $U_n(t)>0$ for all $t>t_n$, and with $t_n\to -\infty$, then there exists a global solution $U$ on $\mathbb{R}$ with $U(t)\ge0$ for all $t$.
  \item Similarly, if there is a sequence $\{V_n\}$ with asymptotes $s_n\to +\infty$ and $V_n(t)<0$ for all $t<s_n$, then there exists a global solution $V$ with $V(t)\le0$ for all $t$.
\end{itemize}
\end{lema}

\begin{proof}
Let $\{U_n\}$ satisfy the first hypothesis. By Proposition~\ref{lema da cota fora da vizinhança da assintota}, for any fixed $a>0$ there is a uniform $C$ so that
$$
\|U_n(t)\|<C
\quad
\text{whenever }t\ge t_n+a.
$$
Since 
$$
U_n'=-U_n^2 - K,
$$
we also get a uniform bound on $\|U_n'(t)\|$ for $t\ge t_n+a$. Hence $\{U_n\}$ is equicontinuous and uniformly bounded on each compact interval of $(t_n+a,\infty)$.

By Arzelà–Ascoli, a subsequence converges uniformly on compacts to a continuous limit $U$ that satisfies the Riccati equation and $U(t)\ge0$ for all $t$.

The same argument holds for $\{V_n\}$ when reversing the time and inequality signs.
\end{proof}

Now, let us analyze the relation between Anosov metrics and the Riccati equation. We start with the standard definition and a classical characterization.

\begin{defi}
Let $(M,g)$ be a compact Riemannian manifold, $\phi_t$ its geodesic flow, and $G$ the geodesic vector field. We say $g$ is an \emph{Anosov metric} if there exist constants $C,\lambda>0$ such that, for every $\theta\in T_1M$, 
$$
T_{\theta}T_1M \;=\; E^s(\theta)\oplus E^u(\theta)\oplus \langle G(\theta)\rangle,
$$
and
$$
\|d\phi_t(\xi^s)\|\le C\,e^{-\lambda t}
\quad(\xi^s\in E^s,\;t\ge0),
\qquad
\|d\phi_t(\xi^u)\|\ge C\,e^{\lambda t}
\quad(\xi^u\in E^u,\;t\le0).
$$
\end{defi}

Anosov \cite{anosov1969geodesic} showed that any compact manifold of negative curvature admits an Anosov metric. Eberlein \cite{eberlein1973geodesic} then proved the following equivalent conditions:

\begin{teo}[\cite{eberlein1973geodesic}, Theorem 3.2]\label{teo equivalencia anosov e jacobi}
Let $M$ be a compact manifold without conjugate points. The following are equivalent:
\begin{enumerate}
  \item The geodesic flow on $T_1M$ is Anosov.
  \item $E^s(\theta)\cap E^u(\theta)=\{0\}$ for all $\theta\in T_1M$.
  \item $T_{\theta}T_1M = E^s(\theta)\oplus E^u(\theta)\oplus \langle G(\theta)\rangle$ for all $\theta$.
  \item Every nonzero Jacobi field along any geodesic is unbounded in norm.
\end{enumerate}
\end{teo}

In the case of surfaces, condition (2) is equivalent to the following:

\begin{lema}
Let $(M,g)$ be a compact surface without conjugate points. Then $g$ is Anosov if and only if 
$$
U^s(t) < U^u(t)
\quad
\text{for all }t\in\mathbb{R}.
$$
\end{lema}

\section{Surfaces free of focal points and with Anosov metric}\label{seção 3}

Inspired by Gulliver’s surgery constructions in \cite{gulliver1975variety}, which produced compact Riemannian surfaces with regions of positive curvature whose geodesic flow is Anosov (with or without focal points according to the surgery), we now introduce a family of surfaces that, with the hypotheses of Theorem \ref{Teorema 1 introdução}, consists of Anosov metrics without focal points. 

\begin{defi}\label{definição da família}
  Consider $\delta, \epsilon, \Lambda \in \mathbb{R}$ such that $\delta > 0$, $\Lambda > 0$, and $0 < \epsilon < 1$, and consider $k \in \mathbb{N}$. We denote by $\mathcal{M}(\delta, k, \epsilon, \Lambda)$ the family of compact, orientable, smooth Riemannian surfaces of genus greater than $1$ such that, if $(M,g) \in \mathcal{M}(\delta, k, \epsilon, \Lambda)$, then:
\begin{enumerate}
\item There exist $k$ pairwise disjoint, simply connected, strongly convex open balls in $M$ that are free of focal points, denoted by $B_{\delta}(p_{1}),\cdots, B_{\delta}(p_{k})$. We shall refer to them as \textbf{generalized bubbles}.
\item Every point where the curvature is non-negative is contained in $\bigcup_{i=1}^{k}B_{\delta}(p_{i})$.
\item For all $ i,j = 1, \dots, k $ with $i\ne j$, the distance between $B_{\delta}(p_{i})$ and $B_{\delta}(p_{j})$, as well as the return time of a geodesic to $B_{\delta}(p_{i})$, is greater than $\Lambda$.
\item The curvature in $(\bigcup_{i=1}^{k}B_{\delta
}(p_{i}))^{c}$ is smaller than $-\epsilon$.

    \end{enumerate}
\end{defi}

For any parameters, a surface $M \in \mathcal{M}(\delta, k, \epsilon, \Lambda)$ exists. Volume, thus $\Lambda$, can be arbitrarily large by increasing genus. Positive curvature can be introduced into negative curvature surfaces via surgeries/perturbations.

\begin{obs}
Being strongly convex, simply connected, and free of focal points, generalized bubbles cannot contain entire geodesics; every geodesic intersecting a generalized bubble necessarily crosses it. Geodesic segments in a generalized bubble have a maximum length of $2\delta$. A longer segment would imply a shorter one connecting its ends, contradicting the absence of focal points. Finally, there cannot be closed geodesics in the generalized bubble, as the region is simply connected and has no focal points.
\end{obs}

\begin{obs}
Definition \ref{definição da família} (3) is equivalent to the distance between lifts of generalized bubbles in the universal covering being greater than $\Lambda$. More precisely, for $i=1,\dots,k$, if $\tilde{B_{\delta}}^{j}(p_{i})$ is a connected component of $B_{\delta}(p_{i})$'s lift, then assumption (3) means $\text{dist}(\tilde{B^{j}_{\delta}}(p_{i}), \tilde{B^{l}_{\delta}}(p_{m})) > \Lambda$ for $j \neq l$ or $i\neq m$.
\end{obs}

The key idea in the proof of Theorem~\ref{Teorema 1 introdução} is to control the dynamical behavior of geodesics inside and outside the generalized bubbles via solutions of the Riccati equation.  Fix a surface $(M,g)$ and denote by $K$ its sectional curvature and by $K^+ = \max_M K$ its maximum.

\begin{lema}\label{lema controlhe fora da bolha}
Let $(M,g)\in \mathcal{M}(\delta,k,\epsilon,\Lambda)$, and let $\gamma$ be a geodesic.  Let $
\bigl\{(a_m,b_m)\bigr\}_{m\in I\subset\mathbb{Z}}
$ be the collection of parameter intervals on which $\gamma(t)$ is in $\bigcup_{i=1}^k B_\delta(p_i)$.  Let $U(t)$ be any Riccati solution along $\gamma$. 
There exists $\Lambda(\epsilon)>0$ such that if $U(b_m)>0$ and $\Lambda\ge\Lambda(\epsilon)$ then
\begin{equation*}
U(t)>0\quad\forall\,t\in[b_m,a_{m+1}],\quad
U\bigl(a_{m+1}\bigr)>\sqrt{\epsilon}\,\Bigl(1-\tfrac{\epsilon}{2}\Bigr).
\end{equation*}
\end{lema}

\begin{proof}
First, by the negativity of the curvature outside the generalized bubble, $U(b_m)>0$ implies 
$$
U(t)>0
\quad\forall\,t\in[b_m,a_{m+1}].
$$

Now, compare $U(t)$ with the solution $V(t)$ of the scalar Riccati equation
\begin{align}
U'(t) + U(t)^2 + K(t) &= 0,\label{eq:U_riccati}\\
V'(t) + V(t)^2 - \epsilon &= 0,\label{eq:V_riccati}
\end{align}
with initial condition $V(b_m)=U(b_m)>0$.  Since $K(t)\le -\epsilon$ on the complement of the generalized bubbles, subtracting \eqref{eq:U_riccati} from \eqref{eq:V_riccati} at $t=b_m$ gives
$$
U'(b_m)-V'(b_m)
= -K(b_m)-\epsilon > 0,
$$
so $U(t)>V(t)$ for all $t\in[b_m,a_{m+1}]$.

The general solutions of \eqref{eq:V_riccati} are
$$
V(t)
=\frac{\sqrt{\epsilon}\,e^{2\sqrt{\epsilon}t} + \sqrt{\epsilon}\,C}
     {e^{2\sqrt{\epsilon}t} - C},
$$
where $C$ is determined by $V(b_m)$.  Since $V(b_m)>0$, $V(t)>\sqrt{\epsilon}$ for all $t>b_m$ or $\lim_{t\to\infty}V(t)=\sqrt{\epsilon}$. Therefore, there exists $\Lambda(\epsilon)>0$ such that if $a_{m+1}-b_m\ge\Lambda(\epsilon)$ then
$$
V(a_{m+1})>\sqrt{\epsilon}\Bigl(1-\frac{\epsilon}{2}\Bigr).
$$
Since $U(a_{m+1})>V(a_{m+1})$, the same lower bound holds for $U(a_{m+1})$.
\end{proof}

\begin{figure}[H]
    \centering
    \includegraphics[scale=0.42]{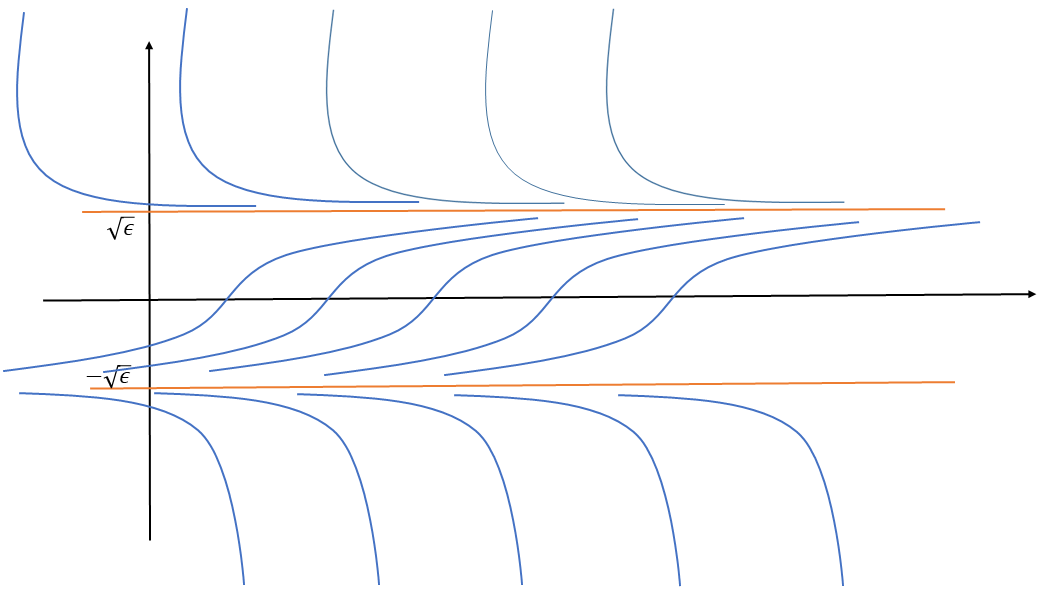}
    \caption{Solutions of the Riccati equation with constant curvature $-\epsilon$.}
\end{figure}

\begin{figure}[H]\label{figura da solução da equação de Riccati fora da bolha}
    \centering
    \includegraphics[scale=0.5]{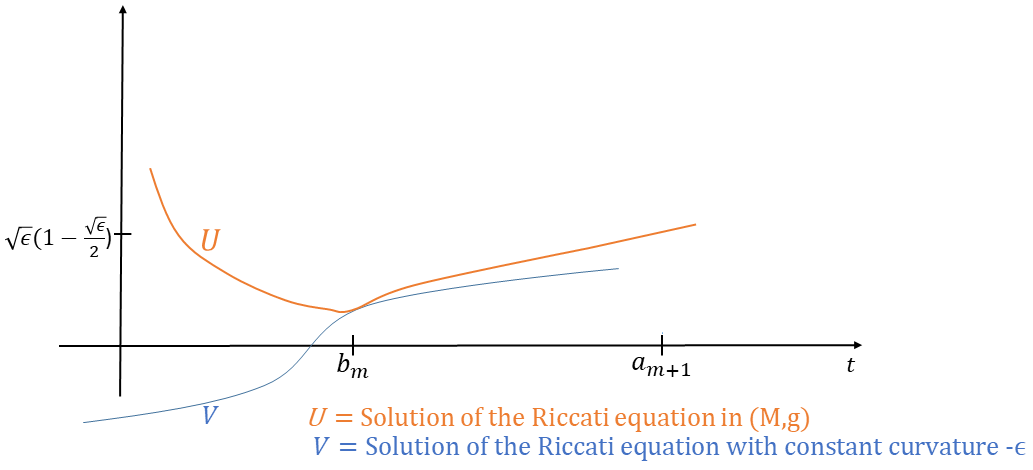}
    \caption{Control of the Riccati solution outside the generalized bubble.}
\end{figure}

\begin{lema}\label{lema sai positivo da bolha e chega grande na próxima}

Consider $(M,g) \in \mathcal{M}(\delta, k, \epsilon, \Lambda)$. Let $\gamma$ be a geodesic and $\{(a_{m},b_{m})\}_{m\in I\subset\mathbb{Z}}$ its intersections with $\bigcup_{i=1}^{k}B_{\delta}(p_{i})$.

There exist $\Lambda(\epsilon) > 0$ such that if $\Lambda \ge \Lambda(\epsilon)$ and
$$K^{+} < \frac{\sqrt{\epsilon}(1-\frac{\epsilon}{2})}{2\delta} - \epsilon(1-\frac{\epsilon}{2}) ^{2},$$
then for all $m$, a Riccati solution $U(t)$ exists on $\gamma$ with $U(t)\ge0$ for $t\in[a_{m},\infty)$. Moreover, any Riccati solution $\ge \sqrt{\epsilon}\big(1-\frac{\epsilon}{2}\big)$ on a generalized bubble is positive on that bubble.

Furthermore, if
$$K^{+} < \frac{\sqrt{\epsilon}\left ( 1-\frac{\epsilon}{2}\right)-\epsilon\left(2\delta+1\right)\left(1 -\frac{\epsilon}{2}\right)^{2}}{2\delta},$$
there exist $U(t)\ge\epsilon(1-\frac{\epsilon}{2})^{2}$ for all $t\in[a_{m},\infty)$.
\end{lema}

\begin{proof}

By Lemma \ref{lema controlhe fora da bolha}, for each $\epsilon$, there exists $\Lambda(\epsilon)$ such that if $\Lambda \ge \Lambda(\epsilon)$, any positive Riccati solution when a geodesic leaves a generalized bubble remains greater than $\sqrt{\epsilon}(1-\frac{\epsilon}{2})$ when entering another. Furthermore, $b_m-a_m < 2\delta$ for all $m$.

Fix $m$ and let $U$ be a Riccati solution with $U(b_{m-1})>0$. Then $U(a_{m})>\sqrt{\epsilon}(1-\frac{\epsilon}{2})$. Assume, without loss of generality, there exists $c\in(a_{m}, b_{m})$ such that $U(c)=\sqrt{\epsilon}(1-\frac{\epsilon}{2})$.

Considering the Riccati equation $U' + U^2 + K = 0$, we have $U' = -U^2 - K$. Integrating from $c$ to $t$:
\begin{align*}
U(t) &= U(c) - \int_{c}^{t} (U^2(s) + K(s)) \, ds \\
& \ge U(c) - \int_{c}^{t} (\max_{[c,b_{m}]} U(s))^2 \, ds - \int_{c}^{t} (\max_{[c,b_{m}]} K(s)) \, ds \\ 
&\ge \sqrt{\epsilon}(1-\frac{\epsilon}{2}) - \int_{c}^{t} (\epsilon(1-\frac{\epsilon}{2})^{2} + K^{+}) \, ds \\
&= \sqrt{\epsilon}(1-\frac{\epsilon}{2}) - (\epsilon(1-\frac{\epsilon}{2})^{2} + K^{+}) (t - c).
\end{align*}
If 
$$
K^{+} < \frac{\sqrt{\epsilon}\left(1 - \frac{\epsilon}{2}\right)}{2\delta} - \epsilon\left(1 - \frac{\epsilon}{2}\right)^2,
$$ 
then $U(t) > 0$ for all $t \in [a_m, b_m]$, and by induction, $U(t) \ge 0$ for all $t \ge a_m$.

Moreover, if
$$
K^{+} < \frac{ \sqrt{\epsilon}\left(1 - \frac{\epsilon}{2}\right) - \epsilon(2\delta + 1)\left(1 - \frac{\epsilon}{2}\right)^2 }{2\delta},
$$
then $U(t) > \epsilon\left(1 - \frac{\epsilon}{2}\right)^2$ on $[a_m, b_m]$, and thus $U(t) \ge \epsilon\left(1 - \frac{\epsilon}{2}\right)^2$ for all $t \ge a_m$.
\end{proof}

\begin{figure}[H]
    \centering
    \includegraphics[scale=0.42]{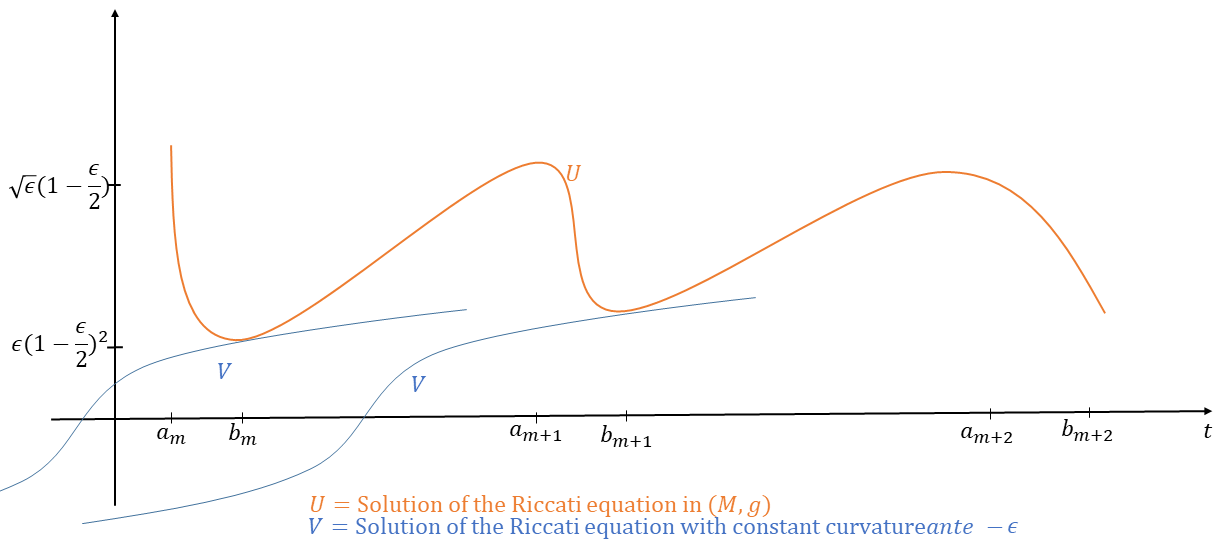}
    \caption{Control of Riccati solution inside and outside the generalized bubble.}
   
    \end{figure}

Now, based on the results shown, we have completed the proof of Theorem~\ref{Teorema 1 introdução}.

\begin{proof}
We prove Theorem~\ref{Teorema 1 introdução}.

Let $\Lambda$ be as in Lemma \ref{lema controlhe fora da bolha}, i.e., $\Lambda > 0$ such that if
$$
V'(t) + V^2(t) - \epsilon = 0 \quad \text{and} \quad V(0) > 0,
$$
then $V(\Lambda) > \sqrt{\epsilon}\left(1 - \frac{\epsilon}{2}\right)$. 

Let $\gamma$ be a geodesic. If $\gamma$ does not intersect a generalized bubble, i.e., it remains entirely in the region of negative curvature, then there exists a solution to the Riccati equation that never vanishes along $\gamma$. The same holds if $\gamma$ is tangent to a bubble.

By construction, any geodesic intersecting a generalized bubble must have defined entry and exit times. Let $\{(a_m, b_m)\}_{m \in I}$ be the sequence of intervals where $\gamma$ intersects the bubbles. We know that $b_m - a_m < 2\delta$. 

Let $J$ be a nontrivial Jacobi field along $\gamma$ with $J(b_1) = 0$, and let $U$ be the associated Riccati solution. Then $U(t) > 0$ for $t \in (b_1, a_2)$, and by Lemma~\ref{lema sai positivo da bolha e chega grande na próxima}, we have
$$
U(t) > \epsilon\left(1 - \frac{\epsilon}{2}\right)^2 \quad \text{for all } t \in (b_1,+\infty).
$$

This is sufficient to construct a global non-negative solution to the Riccati equation. In particular, for any $t_n \in \mathbb{R}$ and $n \in \mathbb{N}$, there exists a solution $U_n$ such that
$$
U_n(t) > \epsilon\left(1 - \frac{\epsilon}{2}\right)^2 \quad \text{for all } t \in (t_n,+\infty).
$$

Let $U_n$ be the solution with asymptote at $t_n$. Then, by Lemma~\ref{lema da convergência da solução da equação de riccati não negativa}, there exists a solution $U$ such that $U(t) \ge 0$ for all $t \in \mathbb{R}$. Since the unstable solution is the supremum of such solutions, it follows that the unstable Riccati solution is non-negative. 

A similar argument using Lemma~\ref{lema da convergência da solução da equação de riccati não negativa} can be made to show that the stable solution of the Riccati equation is non-positive. Therefore, as is true for all geodesics, by Lemma \ref{lema instável não negativo e estável não positivo}, the surface has no focal points.

In particular, if
$$
K^+ < \frac{\sqrt{\epsilon}\left(1 - \frac{\epsilon}{2}\right) - \epsilon(2\delta + 1)\left(1 - \frac{\epsilon}{2}\right)^2}{2\delta},
$$
then along any geodesic there is a solution of the Riccati equation defined for all $t$ with lower bound $\epsilon(1-\frac{\epsilon}{2})^{2}$ and a solution of the Riccati equation defined for all $t$ with upper bound $-\epsilon(1-\frac{\epsilon}{2})^{2}$. The argument is analogous to that in Lemma~\ref{lema da convergência da solução da equação de riccati não negativa}, but now using the modulus of the solutions. 

Let us now prove that $g$ is an Anosov metric.

Since $g$ has no conjugate points, stable and unstable Jacobi fields exist globally. By Theorem~\ref{teo equivalencia anosov e jacobi}, it is sufficient to show that no Jacobi field is simultaneously stable and unstable. This would only be possible if the corresponding geodesic entered a region of non-negative curvature. 

However, by our assumptions, along any geodesic the solution of the unstable Riccati equation is positive and the solution of the stable Riccati equation is negative. This implies that these fields are not generated by the same Jacobi field, i.e., there is no stable and unstable Jacobi field simultaneously. And this proves the Theorem.
\end{proof}

We conclude this section with an observation regarding the choice of the parameter $\Lambda$.

\begin{obs}\label{Obs tamanho de Lambda}
Although we do not provide an explicit value for $\Lambda$, we can estimate its lower bound in terms of $\epsilon$. Indeed, consider the solution to the Riccati equation on a compact surface of constant curvature $-\epsilon$ that vanishes at $t=0$. This solution is given by
$$\sqrt{\epsilon}\tanh(\sqrt{\epsilon}t).$$

Now, let $\gamma:[0,\Lambda] \to M$ be a geodesic segment in $(M,g)$ that is outside the generalized bubbles. If $U$ is a solution to the Riccati equation in $\gamma$ with $U(0) \ge 0$, then
$$U(t) \ge \sqrt{\epsilon}\tanh(\sqrt{\epsilon}t) \quad \text{for all } t \in [0,\Lambda].$$

Therefore, to guarantee that $U(\Lambda) \ge \sqrt{\epsilon}(1 - \frac{\epsilon}{2})$, it is enough to require
$$\sqrt{\epsilon}(1 - \tfrac{\epsilon}{2}) \le \sqrt{\epsilon}\tanh(\sqrt{\epsilon}\Lambda),$$
which is equivalent to
$$\artanh(1 - \tfrac{\epsilon}{2}) \le \sqrt{\epsilon}\Lambda.$$

Hence, we may define $\Lambda$ in Theorem~\ref{Teorema 1 introdução} as
$$\Lambda := \frac{1}{\sqrt{\epsilon}} \artanh\left(1 - \frac{\epsilon}{2}\right).$$
\end{obs}

\section{One parameter conformal deformations of metrics}\label{seção 4}

In this section, we present a metric deformation method and analyze how this deformation influences the estimates on geodesics in the deformed metrics.

We begin by recalling a classical result concerning flows, which is standard in the literature (see, for example, \cite{chang2005conformal}).

\begin{teo}\label{teorema fórmula da curvatura pelo laplaciano}
   Let $(M,g)$ be a smooth compact Riemannian surface. Consider $w \in C^{\infty}(M)$ and $g_{w} = e^{2w} g$. Then
$$-\triangle_{g} w + K_{g} = K_{g_{w}} e^{2w},$$
where $K_{g}$ and $K_{g_{w}}$ are the sectional curvatures of $g$ and $g_{w}$, respectively.
\end{teo}

Another result we are using in this section is the following:
\begin{teo}[\cite{aubin2012nonlinear}, Theorem 4.7]\label{teorema inverso do laplaciano}
    Consider $(M,g)$ a compact $C^{\infty}$ manifold and $h$ a $C^{\infty}$ function. Then there exists $w \in C^{\infty}(M)$ such that $h = \Delta_{g} w$ if and only if $\int h \, d\sigma_{g} = 0$.
\end{teo}

When defining a conformal metric curve, it becomes essential to understand which metrics along that curve preserve certain geometric properties. With this in mind, we present the following definition:

\begin{defi}\label{Definição da segunda família}
    Consider $\overline{\Lambda}, \zeta \in \mathbb{R}$ such that $\zeta > 0$ and $\overline{\Lambda} > 0$. Suppose that $(M,g) \in \mathcal{M}(\delta, k, \epsilon, \Lambda)$, where $B_{\delta}(p_{1}), \dots, B_{\delta}(p_{k})$ are the generalized bubbles of $(M,g)$.  

Given a family of metrics conformal to $g$, denoted by $g_{\rho}$ and parameterized by $\rho \in [0,1]$ with $g_{0} = g$, we define $\mathcal{M}_{\rho}(M,g,\delta, k,\zeta,\overline{\Lambda})$ as the set of metrics $g_{\rho_{l}}$ for $\rho_{l} \in [0,1]$ such that
\begin{enumerate}

\item Every point where the curvature of $(M,g_{\rho_{l}})$ is non-negative is contained in $\bigcup_{i=1}^{k}B_{\delta}(p_{i})$.
\item For $i = 1, \dots, k$, let $\tilde{B_{\delta}}^{j}(p_{i})$ be a connected component of the lift of $B_{\delta}(p_{i})$ in the universal cover of $M$, denoted by $\tilde{M}$.
The distance between any two such balls with respect to the metric $g_{\rho_l}$ is greater than $\overline{\Lambda}$.
\item The curvature of $(M,g_{\rho_{l}})$ in  $(\bigcup_{i=1}^{k}B_{\delta
}(p_{i}))^{c}$ is smaller than $-\zeta$.

    \end{enumerate}
\end{defi}

Note that $B_{\delta}(p_{i})$ represents a ball of radius $\delta$ centered at $p_{i}$ only in the original metric, $g$. As the metric changes, the notion of distance also changes. Denote by $B_{\delta}^{\rho_{l}}(p_{i})$ a ball in the metric $g_{\rho_{l}}$. 

With definition \ref{Definição da segunda família} in mind, we will be able to prove, using the results developed in this section, the following theorem:

\begin{teo}\label{Teorema 2 introdução}
Let $(M,g)\in\mathcal{M}(\delta,k,\epsilon,\Lambda)$ with $\epsilon<\frac{-2\pi\chi(M)}{\operatorname{vol}(M)}
$.

Then there exists $w\in C^{\infty}(M)$ satisfying $\min_{M} w=0$ such that, defining
$$
g_{\rho}:=e^{2\rho w}g,\qquad \rho\in[0,1],
$$
and setting $\mu:=\max_{M} w$, the following properties hold:
\begin{enumerate}
\item for every $\rho_{l}\in[0,1]$,
$$
g_{\rho_{l}}\in\mathcal{M}_{\rho}(M,g,\delta,k,e^{-2\mu}\epsilon,\Lambda);
$$
\item the metric $(M,g_{1})$ has strictly negative curvature;
\item for each $i=1,\dots,k$,
$$
B_{\delta}(p_{i}) \subset B^{\rho_{l}}_{\delta e^{\rho_{l}\mu}}(p_{i}).
$$
\end{enumerate}
\end{teo}

Although it is not one of the main objectives of this work, it succinctly captures the behavior of the deformation introduced in this section.

\medskip

The following result introduces a family of conformal deformations that will be instrumental in the proofs of Theorems~\ref{Teorema 2 introdução} and~\ref{Teorema 3 introdução}. Beyond its statement, the proof also provides interisting insight.

\begin{prop}\label{proposição da deformação}
    
   Consider $(M,g) \in \mathcal{M}(\delta, k, \epsilon, \Lambda)$, with $\epsilon<\frac{-2\pi\chi(M)}{vol(M)}$. Then, there exists $w \in C^{\infty}(M)$ satisfying $\min_{M} w = 0$ such that the deformation  
$$ g_{\rho} := e^{2\rho w} g, \quad \text{with } \rho \in [0,1], $$  
reduces the curvature of points with positive curvature as $\rho$ increases, preserves the negative curvature of points that already have negative curvature,  
and ensures that the negative curvature remains smaller than $-\zeta$ in $(\bigcup_{i=1}^{k} B_{\delta}(p_{i}))^{c}$, where $\zeta = e^{-2\mu} \epsilon$ and $\mu$ is the maximum of $w$ in $M$. \\
    
    Moreover, when $\rho=1$ the curvature of $g_{\rho}$ is strictly negative in $M$.
\end{prop}

\begin{proof}
    First, observe that if $\int_M h\,dv=0$ for some $h\in C^\infty(M)$, then, by Theorem \ref{teorema inverso do laplaciano}, there exists $w\in C^\infty(M)$ with $\Delta_g w=h$.  We will choose
$$
h:=-\Bigl(K - \frac{2\pi\chi(M)}{\operatorname{vol}(M)}\Bigr).
$$
By Gauss–Bonnet Theorem
$$\int_{M}K \, dv= 2\pi\chi(M)\le -4\pi,$$
thus
$$
\int_M h\,dv
=-\int_M K\,dv +2\pi\chi(M)=0.
$$
Then, there exist $w$ such that $\Delta_g w=-h$ and knowing that adding a constant to a function does not change its Laplacian, we can suppose without loss of generality, that $\min_{M}w= 0$. Then for each $\rho\in[0,1]$, by Theorem \ref{teorema fórmula da curvatura pelo laplaciano}, the conformal curvature satisfies
$$
K_{\rho w}=e^{-2\rho w}\bigl(K -\rho\,\Delta_g w\bigr)
=e^{-2\rho w}\bigl(K+\rho\,h\bigr).
$$
Choose $\mu=\max_M w$. 

Therefore, since
  $\frac{2\pi\chi(M)}{vol(M)}<-\epsilon$, for all $\rho\in[0,1]$, the conformal curvature satisfies
\begin{enumerate}
  \item If $x\notin\bigcup_iB_\delta(p_i)$, then $K(x)<-\epsilon$ and 
  $$K_{\rho w}(x)<e^{-2\rho w(x)}(-\epsilon)\le -e^{-2\mu}\epsilon$$
  because 
  $$K(x)+\rho h(x)=K(x)(1-\rho) + \rho\frac{2\pi\chi(M)}{vol(M)}<-\epsilon(1-\rho)-\epsilon\rho=-\epsilon.$$
  \item If $K(x)<0$, then 
  $$K_{\rho w}(x)=e^{-2\rho w(x)}(-\rho \triangle_{g} w(x) + K(x))<0,$$
  because
  $$K(x)+\rho h(x)=K(x)(1-\rho) + \rho\frac{2\pi\chi(M)}{vol(M)}<0$$
  and $e^{-2\rho w(x)}\ge 1$.
  \item If $K(x)\ge0$, then  
  $$K_{\rho w}(x)<e^{-2\rho w(x)}K(x)\le K(x),$$
  because 
  $$K(x)+\rho h(x)=K(x)(1-\rho) + \rho\frac{2\pi\chi(M)}{vol(M)}<K(x)$$
  and $e^{-2\rho w(x)}\ge 1$.
   \item In particular, $K_{\rho w}(x)<0$ if $\rho=1$ since $-\triangle_{g} w(x) + K(x)<0$.
\end{enumerate}

In other words, the curvature function decreases in the regions of positive curvature as $\rho$ increases and maintains a negative upper bound in $(\bigcup_{i=1}^{k}B_{\delta
}(p_{i}))^{c}$. Furthermore, when $\rho=1$ the curvature is strictly negative. This completes the proof.  
\end{proof}

\begin{figure}[H]
    \centering
    \includegraphics[scale=0.52]{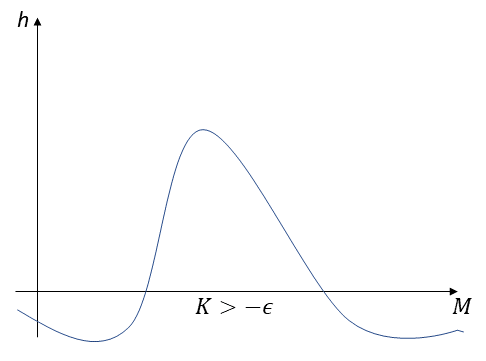}
    \caption{Behavior of $h$.}
   
    \end{figure}

We have seen that the condition $\epsilon < \frac{-2\pi\chi(M)}{\operatorname{vol}(M)}$ is crucial for the construction of the deformation. However this choice does not significantly restrict the initial selection of $\epsilon$.

\medskip

A notable example is given by Gulliver-type surfaces. These are obtained by modifying a surface of negative curvature, removing a small strongly convex ball, and inserting in its place a region with nonnegative curvature, while maintaining the original metric at all other points.

Consider, for instance, a surface $(N,g)$ of genus greater than $1$ with constant curvature $-1$. Then,
$$
\int_{N} K = -\operatorname{vol}_{g}(N) = 2\pi\chi(N) \quad \Longrightarrow \quad -\frac{2\pi\chi(N)}{\operatorname{vol}_{g}(N)} = 1.
$$

Performing the surgery described in \cite{gulliver1975variety}, we obtain a new metric $\overline{g_\delta}$ with curvature $K \equiv -1$ outside a small neighborhood $B$, and with $K \geq 0$ inside $B$. For $\delta$ sufficiently small, we have
$$
\operatorname{vol}_{\overline{g_\delta}}(N) \approx \operatorname{vol}_g(N),
$$
which implies that $\epsilon$ can be chosen arbitrarily close to $1$.

This construction can be applied to any surface with negative constant curvature.\\

When we change the metric, we inevitably change the geodesics, Jacobi fields, curvatures, and other important features. However, it is crucial to estimate these changes. Consider $(M,g)\in \mathcal{M}(\delta, k,\epsilon,\Lambda)$ and $w$ as in Proposition \ref{proposição da deformação}.

The deformation we have constructed makes it possible to estimate the length of the curves.

\begin{lema}\label{lema geodésica aumenta}
  Let $\sigma: [a, b] \to M$ be a smooth curve. Then the length of $\sigma$ in the $g$ metric is smaller than (or equal to) the length of $\sigma$ in the $g_{\rho}$ metric.
\end{lema}

\begin{proof}
In fact, since $g_{\rho}=e^{2\rho w}$ with $w\ge 0$, we have
$$g_{\rho}(X,X)\ge g(X,X) \, \forall x\in TM.$$
Then,
$$L(\sigma) - L_{\rho}(\sigma) = \int_{a}^{b} ||\sigma'(t)||_{g} \, dt - \int_{a}^{b} ||\sigma'(t)||_{g_{\rho}} \, dt\le 0.$$
\end{proof}

\begin{coro}\label{corolario a  distancia entre  pontos aumenta}
   The distance between points in $M$ does not decrease as $\rho$ increases. 
\end{coro}

\begin{proof}
Let $p$ and $q$ be points in $M$. Let $\gamma_{\rho}$ be the geodesic segment connecting $p$ and $q$ in the metric $g_{\rho}$. By Lemma \ref{lema geodésica aumenta},
$$L(\gamma_{\rho})\le L_{\rho}(\gamma_{\rho}).$$
Therefore, since the distance between $p$ and $q$ in the metric $g$ is less than $L(\gamma_{\rho})$ the result follows.
\end{proof}

\begin{obs}
Although the region of non-negative curvature is decreasing as a set of points during the conformal deformation, this does not guarantee that the maximum distance between points with non-negative curvature is also decreasing. In fact, the distance between such points may increase in their respective metrics, as we show in Corollary~\ref{corolario a  distancia entre  pontos aumenta}.
\end{obs}

Let us now verify how the distance between the generalized bubbles behaves as the parameter $\rho$ increases.

\begin{lema}\label{lema distância entre as bolhas aumenta}
Consider $\tilde{B}_{\delta}^{j_{1}}(p_{i_{1}}), \tilde{B}_{\delta}^{j_{2}}(p_{i_{2}}) \subset (\tilde{M}, \tilde{g})$ generalized bubbles, lifting of $(M,g)$, such that $i_{1} \ne i_2$ or $j_{1} \ne j_2$. Then, for $\rho_{1}, \rho_{2} \in [0,1]$ with $\rho_{2} > \rho_{1}$,
$$
\tilde{d}_{\rho_{2}}(\tilde{B}_{\delta}^{j_{1}}(p_{i_{1}}), \tilde{B}_{\delta}^{j_{2}}(p_{i_{2}})) \ge \tilde{d}_{\rho_{1}}(\tilde{B}_{\delta}^{j_{1}}(p_{i_{1}}), \tilde{B}_{\delta}^{j_{2}}(p_{i_{2}})),
$$
where $\tilde{d}_{\rho_{2}}$ and $\tilde{d}_{\rho_{1}}$ are the distance functions associated with the metrics $\tilde{g}_{\rho_{2}}$ and $\tilde{g}_{\rho_{1}}$, respectively.
\end{lema}

\begin{proof}
Given $\rho$, consider $(\tilde{M}, \tilde{g}_{\rho})$ as the covering of $M$ endowed with the pullback of $g_{\rho}$. Since $w \ge 0$, we have
$$
\tilde{g}_{\rho_{2}}(X,X) = e^{2\rho_{2}(w\circ\pi)}\tilde{g}(X,X) \ge e^{2\rho_{1}(w\circ\pi)}\tilde{g}(X,X) = \tilde{g}_{\rho_{1}}(X,X), \quad \forall X \in T\tilde{M}.
$$

Let $q_{1} \in \tilde{B}_{\delta}^{j_{1}}(p_{i_{1}})$ and $q_{2} \in \tilde{B}_{\delta}^{j_{2}}(p_{i_{2}})$, and let $\tilde{\gamma}_{\rho_{2}} : [c,d] \to \tilde{M}$ be the minimizing geodesic segment connecting $q_{1}$ and $q_{2}$. Then,
$$
\tilde{L}_{\rho_{2}}(\tilde{\gamma}_{\rho_{2}}) - \tilde{L}_{\rho_{1}}(\tilde{\gamma}_{\rho_{2}}) = \int_{c}^{d} \left( \|\tilde{\gamma}_{\rho_{2}}'(t)\|_{\tilde{g}_{\rho_2}} - \|\tilde{\gamma}_{\rho_{2}}'(t)\|_{\tilde{g}_{\rho_1}} \right) dt \ge 0.
$$
On the other hand,
$$
\tilde{d}_{\rho_{1}}(q_{1}, q_{2}) \le \tilde{L}_{\rho_1}(\tilde{\gamma}_{\rho_{2}}) \le \tilde{L}_{\rho_2}(\tilde{\gamma}_{\rho_{2}}) = \tilde{d}_{\rho_{2}}(q_{1}, q_{2}).
$$
Since this is valid for any pair of points in $\tilde{B}_{\delta}^{j_{1}}(p_{i_{1}})$ and $\tilde{B}_{\delta}^{j_{2}}(p_{i_{2}})$, the result follows.
\end{proof}

It is also important to estimate the size of a generalized bubble after the metric is deformed.

\begin{lema}\label{lema comprimento da bolha}
$B_{\delta}(q) \subset B^{\rho}_{\delta e^{\rho\mu}}(q)$, where $B^{\rho}_{\delta e^{\rho\mu}}(q)$ denotes the ball centered at $q$ with radius $\delta e^{\rho\mu}$ in the metric $g_{\rho}$. In particular, for all $\rho \in [0,1]$,
$$
B_{\delta}(q) \subset B^{\rho}_{\delta e^{\mu}}(q).
$$
\end{lema}

\begin{proof}
Fix $\rho \in [0,1]$ and let $p \in B_{\delta}(q)$.

Let $\gamma : [a,b] \to M$ be the geodesic segment connecting $q$ and $p$ in the original metric $g$. As in Lemma~\ref{lema geodésica aumenta}, we compute:
$$
L_{\rho}(\gamma) = \int_{a}^{b} \|\gamma'(t)\|_{g_{\rho}} \, dt = \int_{a}^{b} \sqrt{g_{\rho}(\gamma'(t),\gamma'(t))} \, dt = \int_{a}^{b} e^{\rho w(\gamma(t))} \sqrt{g(\gamma'(t),\gamma'(t))} \, dt \le e^{\rho\mu} L(\gamma).
$$

Since $p \in B_{\delta}(q)$, we have $L(\gamma) < \delta$, and thus
$$
L_{\rho}(\gamma) \le \delta e^{\rho\mu}.
$$
Note that $\gamma$ may not be minimizing with respect to $g_{\rho}$, but this inequality ensures
$$
d_{\rho}(p,q) \le \delta e^{\rho\mu},
$$
which implies
$$
B_{\delta}(q) \subset B^{\rho}_{\delta e^{\rho\mu}}(q).
$$
\end{proof}

The previous lemma provides an estimate for the length of generalized bubbles as the metric is deformed, and this estimate will be crucial in the development of the proof of Theorem~\ref{Teorema 3 introdução} in the next chapter.
It is worth noting, however, that although the bound is essential, it presents a subtle difficulty: a priori, when considering the balls $B_{\delta}(q_1)$ and $B_{\delta}(q_2)$, it may happen that
$$
B_{\delta e^{\mu}}^{\rho}(q_1)\cap B_{\delta}(q_2)\neq \varnothing
\quad \text{for some } \rho\in[0,1].
$$
    \begin{figure}[H]
    \centering7
    \includegraphics[scale=0.28]{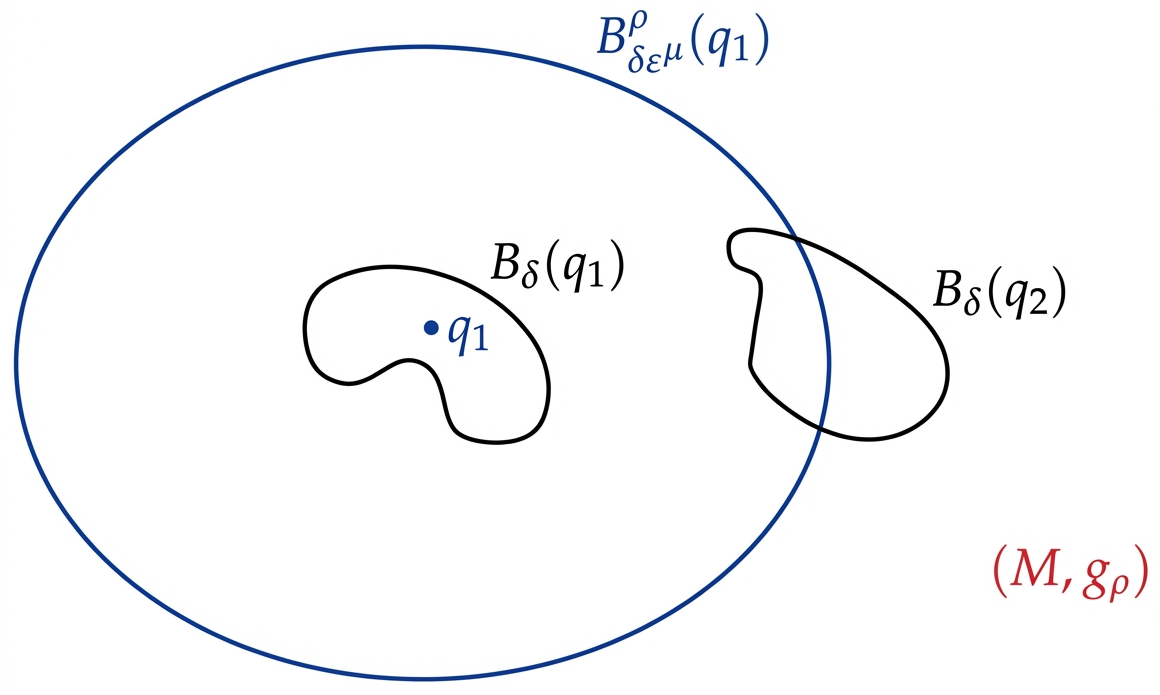}
    \caption{Possible intersection between $B_{\delta e^{\mu}}^{\rho}(q_1)$ and $B_{\delta}(q_2)$ under the metric deformation.}
    \end{figure}
This problem will require delicate work in the following chapter.
For now, let us proceed to conclude the present chapter.

\medskip

We are now ready to prove Theorem~\ref{Teorema 2 introdução}.
\begin{proof}
We prove Theorem~\ref{Teorema 2 introdução}.

By Proposition~\ref{proposição da deformação}, for $i = 1, \cdots, k$, there exists a deformation $g_{\rho} := e^{2\rho w}g$ with $\rho \in [0,1]$ that maintains the curvature smaller than $-e^{-2\mu}\epsilon$ in $(\bigcup_{i=1}^{k} B_{\delta}(p_{i}))^{c}$, the non-negative curvature in $\bigcup_{i=1}^{k} B_{\delta}(p_{i})$ for all $\rho \in [0,1]$, and defines a metric with strictly negative curvature when $\rho = 1$.

By Lemma~\ref{lema comprimento da bolha}, we have
$$
B_{\delta}(p_{i}) \subset B^{\rho}_{\delta e^{\mu}}(p_{i}).
$$
Note that, for all $\rho \in [0,1]$, the curvature of $g_{\rho}$ is smaller than $-e^{-2\mu}\epsilon$ in 
$$
\left( \bigcup_{i=1}^{k} B_{\delta e^{\mu}}(p_{i}) \right)^{c}.
$$

Furthermore, the distance between points does not decrease as $\rho$ increases, by Corollary~\ref{corolario a  distancia entre  pontos aumenta}. On the other hand, since $(M,g) \in \mathcal{M}(M,\delta,k,\epsilon,\Lambda)$, the distance between the lifting of the generalized bubbles is greater than $\Lambda$ in the original metric. Then, by Lemma~\ref{lema distância entre as bolhas aumenta}, for all $\rho \in [0,1]$, the $g_{\rho}$-distance between the lifting of the generalized bubbles is greater than $\Lambda$.

Therefore, 
$$
(M,g_{\rho_{l}}) \in \mathcal{M}_{\rho}(M,B_{\delta}(p_{i}), k, e^{-2\mu}\epsilon, \Lambda),
$$ 
$$
B_{\delta}(p_{i}) \subset B_{\delta e^{\rho_{l}\mu}}^{\rho_{l}}(p_{i}) \quad \text{for all } \rho_{l} \in [0,1],
$$ 
and $(M,g_{1})$ has strictly negative curvature.
\end{proof}

\section{Deformation without focal points and Anosov} \label{seção 5}

In this section, we show sufficient conditions under which not only the original surface has no focal points, but also the surfaces obtained via conformal deformation of the metric remain free of focal points and are Anosov. Our goal is to prove the final result of this paper, namely Theorem~\ref{Teorema 3 introdução}.

\medskip

Before that, it is important to consider the following observation:

\begin{obs}
It is worth noting that there exist surfaces whose metrics admit regions of positive curvature while still satisfying the assumptions of the Theorem \ref{Teorema 3 introdução}. In fact, one can construct examples by considering neighborhoods of metrics on surfaces with nonpositive curvature, where the regions of zero curvature are confined to controlled and well-separated bubbles.

Let us consider the following situation: let $(M,\mathfrak{g})$ be a surface with nonpositive curvature such that
$$
(M,\mathfrak{g}) \in \mathcal{M}\left(\frac{\delta}{2}, k, 2\epsilon, 2\Lambda\right).
$$
Then, if $\delta$ is sufficiently small, there exists a $C^{\infty}$-neighborhood $V_{\mathfrak{g}}$ of the metric $\mathfrak{g}$ such that, for every $g \in V_{\mathfrak{g}}$, we have
$$
(M,g) \in \mathcal{M}(\delta, k, \epsilon, \Lambda),
$$
since the assumptions of strong convexity and absence of focal points for the generalized bubbles are preserved for sufficiently small balls.

On the other hand, consider the function
$$
h_{\mathfrak{g}} := -\left(K_{\mathfrak{g}} - \frac{2\pi\chi(M)}{\operatorname{vol}_{\mathfrak{g}}(M)}\right).
$$
Associated with $h_{\mathfrak{g}}$, we define the function $w_{\mathfrak{g}}$ as in Proposition~\ref{proposição da deformação}. In fact, since both the curvature and the volume depend continuously on the metric, the function
$$
h_{g} := -\left(K_{g} - \frac{2\pi\chi(M)}{\operatorname{vol}_{g}(M)}\right)
$$
also depends continuously on $g$. Therefore, by the definition of $w_g$ as the solution to the inverse Laplacian problem, we have that, in a $C^{\infty}$-neighborhood of $\mathfrak{g}$, the map
$$
F: V_{\mathfrak{g}} \to C^{\infty}(M), \quad g \mapsto w_{g},
$$
where $w_{g} = \Delta_g h_{g}$ with $\min_{M} w_{g} = 0$, is continuous.

Thus, since by hypothesis $K_{\mathfrak{g}} \leq 0$, the surface $(M,\mathfrak{g})$ satisfies the curvature condition in Theorem~\ref{Teorema 3 introdução}, and by continuity, the same holds for all metrics in a neighborhood of $\mathfrak{g}$. Consequently, there exist metrics admitting regions of positive curvature that still satisfy the assumptions of Theorem~\ref{Teorema 3 introdução}.
\end{obs}

Consider $(M,g)\in \mathcal{M}(\delta, k,\epsilon,\Lambda)$ with $\epsilon<\frac{-2\pi\chi(M)}{\operatorname{vol}(M)}$ and
$$
\Lambda = \frac{1}{\sqrt{\epsilon}}\,\mathrm{artanh}\left(1 - \frac{\epsilon}{2}\right),
$$
and let $w$ be as in Proposition~\ref{proposição da deformação}. We denote by $B_{\delta}(p)$ a ball in the metric $g$, and by $B_\delta(p_{1}), \ldots, B_\delta(p_{k})$ the generalized bubbles of $(M,g)$. Their liftings in $(\tilde{M},\tilde{g})$ are denoted by $\tilde{B}^{1}_{\delta}(p_{1}), \ldots, \tilde{B}^{l}_{\delta}(p_{k})$, etc. By Lemma~\ref{lema distância entre as bolhas aumenta}, these generalized bubbles remain at distance greater than $\Lambda$ after the deformation.

Let $B^{\rho}_{r}(p_{i})$ be a ball of radius $r$ in $(M,g_{\rho})$, and $\tilde{B}^{\rho}_{r}(p_{i}^{j})$ its lift in $(\tilde{M},\tilde{g}_{\rho})$, where $p_{i}^{j}$ is a lift of $p_{i}$.

We now establish estimates for the Riccati equation outside the generalized bubbles after deforming the metric. Recall from Remark~\ref{Obs tamanho de Lambda} that we may take
$$
\Lambda = \frac{1}{\sqrt{\epsilon}}\,\artanh\left(1 - \frac{\epsilon}{2}\right)
$$
in Theorem~\ref{Teorema 1 introdução}.

\begin{lema}\label{lema de quanto cresce a solução da equação de Riccati a partir da escolha de Lambda}
Let $\gamma_{\rho}$ be a geodesic in $(M,g_{\rho})$ such that
$$
\gamma_{\rho}|_{[0,\overline{\Lambda}]}:[0,\overline{\Lambda}] \to \left(\bigcup_{i=1}^{k} B_{\delta}(p_{i})\right)^{c}
$$
with $\overline{\Lambda} \geq \Lambda = \frac{1}{\sqrt{\epsilon}}\,\artanh\left(1 - \frac{\epsilon}{2}\right)$. Then, if $U_{\rho}$ is a solution to the Riccati equation along $\gamma_{\rho}$ with $U_{\rho}(0)\ge 0$, we have
$$
U_{\rho}(\overline{\Lambda}) > e^{-\mu}\sqrt{\epsilon}\,\tanh\left(\frac{1}{2}e^{-\mu}\ln(3)\right).
$$
\end{lema}

\begin{proof}
For all $\rho \in [0,1]$, we have $K_{\rho} < -\zeta = -e^{-2\mu} \epsilon$ on 
$$
\left(\bigcup_{i=1}^{k} B_{\delta}(p_{i})\right)^{c}.
$$
By comparison with constant curvature $-\zeta$, we obtain
$$
U_{\rho}(\Lambda) > \sqrt{\zeta} \,\tanh\left(\sqrt{\zeta} \,\Lambda\right) = e^{-\mu} \sqrt{\epsilon} \,\tanh\left(e^{-\mu} \sqrt{\epsilon} \,\Lambda\right).
$$
Substituting the value of $\Lambda$, we get
\begin{align*}
U_{\rho}(\Lambda) 
&> e^{-\mu} \sqrt{\epsilon} \,\tanh\left(e^{-\mu} \sqrt{\epsilon} \cdot \frac{1}{\epsilon} \,\artanh\left(1 - \frac{\epsilon}{2}\right)\right)
> e^{-\mu} \sqrt{\epsilon} \,\tanh\left(e^{-\mu} \,\artanh\left(1 - \frac{\epsilon}{2}\right)\right).
\end{align*}
Since $\epsilon < 1$, we have
$$
\artanh\left(1 - \frac{\epsilon}{2}\right) \geq \artanh\left(\frac{1}{2}\right) = \frac{1}{2}\ln(3),
$$
and thus
$$
U_{\rho}(\Lambda) > e^{-\mu} \sqrt{\epsilon} \,\tanh\left(\frac{1}{2} e^{-\mu} \ln(3)\right).
$$
As $\overline{\Lambda} \geq \Lambda$ and the curvature remains less than $-\zeta$, $U_\rho$ increases further, so
$$
U_{\rho}(\overline{\Lambda}) > e^{-\mu} \sqrt{\epsilon} \,\tanh\left(\frac{1}{2} e^{-\mu} \ln(3)\right).
$$
\end{proof}

When the metric is deformed, the property that a generalized bubble is strongly convex may no longer hold. 
To circumvent this problem, the following proposition ensures that even if a geodesic in a metric $\tilde{g}_{\rho}$ on the universal covering returns to the lifting of a generalized bubble after leaving it, such a return cannot occur too quickly.

\begin{prop}\label{proposição nova}
    Let $\tilde{B}^{\rho}_{r}(p_{i}^{j})$ be a ball in $(\tilde{M},\tilde{g}_{\rho})$ containing $\tilde{B}^{j}_{\delta}(p_{i})$ and disjoint from any other lifting of a generalized bubble. Then any $\tilde{g}_{\rho}$-geodesic goes out from $\tilde{B}^{\rho}_{r}(p_{i}^{j})$ cannot return to $\tilde{B}^{\rho}_{r}(p_{i}^{j})$ without first encircling or intersecting a lifting of a generalized bubble.
\end{prop}

\begin{proof}
Let $\tilde{\gamma}_\rho$ be a $\tilde{g}_{\rho}$--geodesic leaving $\tilde{B}^{\rho}_{r}(p_{i}^{j})$ at $\tilde{\gamma}_\rho(b)$ and suppose, for contradiction, that $\tilde{\gamma}_\rho$ returns to $\tilde{B}^{\rho}_{r}(p_{i}^{j})$ at some point $\tilde{\gamma}_\rho(c)$ without first encircling or intersecting a lifting of a generalized bubble. Then the arc
$$
\big[\tilde{\gamma}_\rho(b),\tilde{\gamma}_\rho(c)\big]
$$
together with the boundary $\partial\tilde{B}^{\rho}_{r}(p_{i}^{j})$ bounds a region $R$ in which the curvature is strictly negative.

Consider a $\tilde{g}_{\rho}$--geodesic $\tilde{\sigma}_1$ which is tangent to $\partial\tilde{B}^{\rho}_{r}(p_{i}^{j})$ at the point $\tilde{\gamma}_\rho(b)$. Starting at that tangency point, $\tilde{\sigma}_1$ must either meet the segment $[\tilde{\gamma}_\rho(b),\tilde{\gamma}_\rho(c)]$ 1n $R$, or re-enter $\tilde{B}^{\rho}_{r}(p_{i}^{j})$ before meeting that segment. The first case cannot occur because, in a region of strictly negative curvature, two distinct geodesics meet at most once; since $\tilde{\sigma}_1$ and $\tilde{\gamma}_\rho$ already meet at $\tilde{\gamma}_\rho(b)$, they cannot intersect again in $R$. Hence the second case must occur, $\tilde{\sigma}_1$ intersects $\tilde{B}^{\rho}_{r}(p_{i}^{j})$ at some point $\tilde{\sigma}_1(x_1)$.

Now repeat the construction: let $\tilde{\sigma}_2$ be the $\tilde{g}_{\rho}$--geodesic tangent to $\partial\tilde{B}^{\rho}_{r}(p_{i}^{j})$ at $\tilde{\sigma}_1(x_1)$. By the same argument, $\tilde{\sigma}_2$ cannot meet the previous segment of $\tilde{\sigma}_1$ in the strictly negative region, then $\tilde{\sigma}_2$ must meet $\tilde{B}^{\rho}_{r}(p_{i}^{j})$ at some point $\tilde{\sigma}_2(x_2)$ with $\tilde{\sigma}_2(x_2)\neq\tilde{\sigma}_1(x_1)$.

Continuing inductively we obtain a sequence of geodesics $(\tilde{\sigma}_n)_{n\ge1}$ tangent to $\partial\tilde{B}^{\rho}_{r}(p_{i}^{j})$ and a sequence of compact regions
$$
R \supset R_1 \supset R_2 \supset \cdots
$$
each bounded by $\partial\tilde{B}^{\rho}_{r}(p_{i}^{j})$ together with a segment of some $\tilde{\sigma}_n$, with each inclusion strict. The intersection of these compact regions is nonempty; in fact its limit must be contained in $\partial\tilde{B}^{\rho}_{r}(p_{i}^{j})$. Consequently there is an accumulation point on $\partial\tilde{B}^{\rho}_{r}(p_{i}^{j})$ which is the limit of the tangency points $\tilde{\sigma}_n(x_n)$. Passing to a subsequence if necessary, the geodesics $\tilde{\sigma}_n$ converge (locally uniformly) to a geodesic $\tilde{\sigma}_\infty$ tangent to $\partial\tilde{B}^{\rho}_{r}(p_{i}^{j})$ at that limit point.

By construction, for $n$ large the geodesic $\tilde{\sigma}_\infty$ must intersect $\tilde{\sigma}_n$ in two distinct points, contradicting the fact that in a region of strictly negative curvature two distinct geodesics cannot have more than one intersection point. 

Therefore any $\tilde{g}_{\rho}$--geodesic leaving $\tilde{B}^{\rho}_{r}(p_{i}^{j})$ must either encircle or intersect some lifting of a generalized bubble before returning to $\tilde{B}^{\rho}_{r}(p_{i}^{j})$.
\end{proof}

\begin{figure}[H]
    \centering
    \includegraphics[scale=0.28]{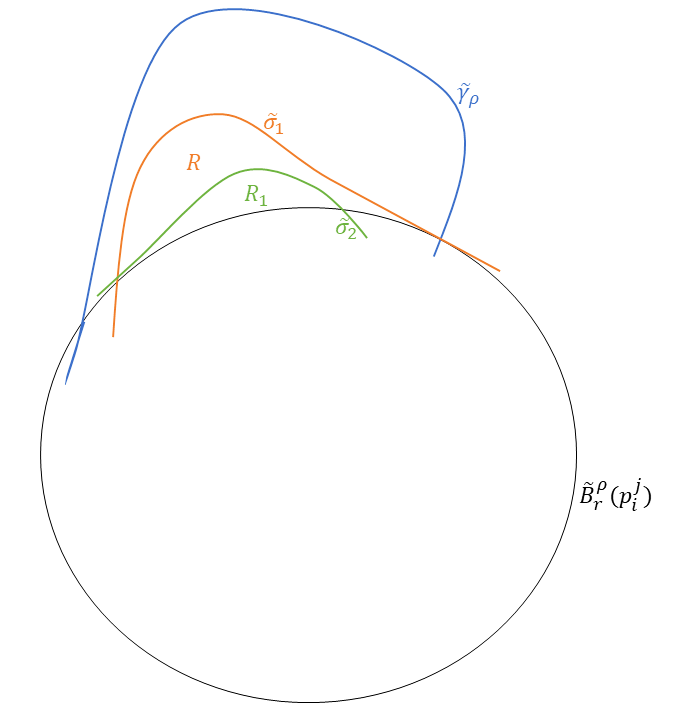}
    \caption{Illustration of the initial configuration of the geodesics and the regions involved in the sequence of arguments used in the proof of Proposition~\ref{proposição nova}.}
   
    \end{figure}
The key point in the proof of the previous lemma is the convexity of the geodesic's behavior in the region where the curvature is negative, because we know that the points of positive curvature are contained in the generalized bubbles. 
The next lemma, in contrast, is an immediate consequence of Lemma~\ref{lema comprimento da bolha}.

\begin{lema}\label{lema sai da bolha e não volta mais depois de um tempo}
Let $(\tilde M,\tilde g_{\rho})$ be the universal cover of $(M,g_{\rho})$ and assume $(M,g_{\rho})$ has no focal points. Fix $r>0$ and denote by $\tilde{B}_{r}^{\rho}(p_{i}^{j})$ the ball of radius $r$ with respect to $\tilde g_{\rho}$ centered at a lifting $p_{i}^{j}$. If $\tilde{\gamma}_{\rho}:\mathbb{R}\to\tilde M$ is a $\tilde g_{\rho}$--geodesic with
$$
\tilde{\gamma}_{\rho}(0)\in\tilde{B}_{r}^{\rho}(p_{i}^{j}),
$$
then
$$
\tilde{\gamma}_{\rho}(t)\notin\tilde{B}_{r}^{\rho}(p_{i}^{j})\qquad\text{for every }t\in\mathbb{R}\setminus(-2r,2r).
$$
In particular, if $\tilde{B}^{j}_{\delta}(p_{i})$ is a connected component of the lifting of a generalized bubble and
$$
\tilde{B}^{j}_{\delta}(p_{i})\subset\tilde{B}_{\delta e^{\mu}}^{\rho}(p_{i}^{j}),
$$
then any $\tilde g_{\rho}$--geodesic $\tilde{\gamma}_{\rho}$ with $\tilde{\gamma}_{\rho}(0)\in\tilde{B}^{j}_{\delta}(p_{i})$ satisfies
$$
\tilde{\gamma}_{\rho}(t)\notin\tilde{B}^{j}_{\delta}(p_{i})\qquad\text{for every }t\in\mathbb{R}\setminus(-2\delta e^{\mu},2\delta e^{\mu}),
$$
and hence, once $\tilde{\gamma}_{\rho}$ leaves $\tilde{B}_{2\delta e^{\mu}}^{\rho}(p_{i}^{j})$, it cannot return to $\tilde{B}^{j}_{\delta}(p_{i})$.
\end{lema}

\begin{proof}
Let $\tilde{\gamma}_{\rho}$ be a $\tilde g_{\rho}$--geodesic with $\tilde{\gamma}_{\rho}(0)\in\tilde{B}_{r}^{\rho}(p_{i}^{j})$. Suppose, for contradiction, that there exists $t\ge 2r$ such that $\tilde{\gamma}_{\rho}(t)\in\tilde{B}_{r}^{\rho}(p_{i}^{j})$ (the case $t\le -2r$ is analogous). Set
$$
\tilde q_1:=\tilde{\gamma}_{\rho}(0),\qquad \tilde q_2:=\tilde{\gamma}_{\rho}(t).
$$
By the triangle inequality there is a piecewise geodesic path from $\tilde q_1$ to $\tilde q_2$ passing through the center $p_{i}^{j}$ whose length is at most
$$
d_{\tilde g_{\rho}}(\tilde q_1,p_{i}^{j})+d_{\tilde g_{\rho}}(p_{i}^{j},\tilde q_2)<r+r=2r.
$$
Hence the length of the minimizing geodesic joining $\tilde q_1$ and $\tilde q_2$ is strictly less than $2r$.

On the other hand, the segment of $\tilde{\gamma}_{\rho}$ between $\tilde q_1$ and $\tilde q_2$ is a geodesic of length $t\ge 2r$. Since $(M,g_{\rho})$ has no focal points (and therefore no conjugate points on the universal cover), minimizing geodesics between two given points are unique; consequently the minimizing geodesic between $\tilde q_1$ and $\tilde q_2$ must coincide with the geodesic segment of $\tilde{\gamma}_{\rho}$. Then $t<2r$, contradicting $t\ge2r$.
\end{proof}

Our next goal is to prove the following proposition, which indeed represents a key step toward establishing Theorem~\ref{Teorema 3 introdução}.

\begin{prop}\label{prop deformação sem pontos locais}
Consider $0<\epsilon<1$, $\delta>0$, and $\Lambda:=\frac{1}{\sqrt{\epsilon}}\,\mathrm{artanh}\left(1-\frac{\epsilon}{2}\right)$.
Suppose that $(M,g)\in \mathcal{M}(\delta, k,\epsilon,\Lambda)$ and that $\epsilon<\frac{-2\pi\chi(M)}{vol(M)}$.
Then, there exists a function $w\in C^{\infty}(M)$ with $\min_{M}w=0$ such that if the maximum curvature $K^{+}$ of $(M,g)$ satisfies
$$
K^{+}
<\frac{\sqrt{\epsilon}}{4e^{2\mu}\,\delta}\left[
  \tanh\!\left(e^{-\mu}\tfrac{1}{3}\ln3\right)
  -\sqrt{\epsilon}\,e^{-\mu}\,\tanh^{2}\!\left(e^{-\mu}\tfrac{1}{3}\ln3\right)\,(4e^{\mu}\delta+1)
\right],
$$
where $\mu:=\max_{M}w$, then the conformal family of metrics
$$
g_{\rho}:=e^{2\rho w}g,\quad \rho\in[0,1],
$$
consists entirely of metrics without focal points.
\end{prop}

To avoid an overly lengthy argument, we begin by proving the following technical lemma, which will play a crucial role in the proof of the proposition.

\begin{lema}\label{lema não tem pontos focais na pertubação}
Consider $\zeta_{1}=e^{-2\mu}\epsilon$. There exists $r>0$ such that if
  \begin{itemize}
      \item $(M,g_{\overline{\rho}})$ is free of focal points,
      \item for each generalized bubble $B_{\delta}(p_{i})$ and $g_{\overline{\rho}}-$geodesic $\gamma_{\overline{\rho}}$ with $\gamma_{\overline{\rho}}(0)\in B_{\delta}(p_{i})$, if $\tilde{\gamma}_{\overline{\rho}}$ is the lifting of $\gamma_{\overline{\rho}}$ such that $\tilde{\gamma}_{\overline{\rho}}(0)\in \tilde{B}^{j}_{\delta}(p_{i})$, and if $(a,b)$ is the connected component of $\tilde{\gamma}_{\overline{\rho}}\cap \tilde{B}^{\rho}_{2\delta e^{\mu}}(p^{j}_{i})$ containing $\tilde{\gamma}_{\overline{\rho}}(0)$,
      if $U$ is a solution of the Riccati equation in $\gamma_{\overline{\rho}}$ such that $U(c)>e^{-\mu}\sqrt{\epsilon}\tanh(e^{-\mu}\frac{1}{3}\ln(3))$ for some $c\in(a,b)$ then $U(t)>e^{-2\mu}\epsilon\tanh^2(e^{-\mu}\frac{1}{3}\ln(3))$ for all $t\in(a,b)$,
  \end{itemize}
  then $(M,g_{\rho})$ is free of focal points if $\rho\in[\overline{\rho},\rho_{1}]$ with $\rho_{1}=\min\{\overline{\rho}+r,1\}$.
\end{lema}

\begin{proof}
Let $\gamma_{\rho}$ be a $g_{\rho}-$geodesic. The proof proceeds in the following stages:

  \begin{enumerate}
      \item \textbf{Intersection of geodesics and generalized bubbles in the universal covering.}
      
      If $\gamma_{\rho}$ does not intersect a generalized bubble, it has no focal points. Suppose $\gamma_{\rho}(s_{1})$ is in a generalized bubble $B_{\delta}(p_{i})$ for some $s_1$. Let $\tilde{\gamma}_{\rho}$ be its lifting in $(\tilde{M},\tilde{g})$ such that $\tilde{\gamma}_{\rho}(s_{1})\in \tilde{B}_{\delta}^{j}(p_{i})$.

      Consider a $g_{\overline{\rho}}-$geodesic $\gamma_{\overline{\rho}}^{s_{1}}$ with initial conditions $$\gamma_{\overline{\rho}}^{s_{1}}(s_{1})=\gamma_{\rho}(s_{1}) \, \text{ and } \, \gamma_{\overline{\rho}}'(x)=\frac{\gamma_{\rho}'(x)}{||\gamma_{\rho}'(x)||_{\overline{\rho}}}.$$
      Let $\tilde{\gamma}^{s_{1}}_{\overline{\rho}}$ be its lifting in $(\tilde{M},\tilde{g_{\overline{\rho}}})$ such that $\tilde{\gamma}_{\rho}(s_{1})=\tilde{\gamma}_{\overline{\rho}}(s_{1})$. Since $(M,g_{\overline{\rho}})$ is free of focal points, $\tilde{\gamma}^{s_{1}}_{\overline{\rho}}$ is not contained in any ball and, by Lemma \ref{lema sai da bolha e não volta mais depois de um tempo}, if it leaves $\tilde{B}^{\overline{\rho}}_{2\delta e^{\mu}}(p_{i}^{j})$, it does not return to $\tilde{B}^{\overline{\rho}}_{\delta e^{\mu}}(p_{i}^{j})$. 
      
      Thus, there exists a unique connected component $\tilde{\gamma}^{s_{1}}_{\overline{\rho}}\big|_{(a,b)}$ of $\tilde{\gamma}^{s_{1}}_{\overline{\rho}}\cap \tilde{B}^{\overline{\rho}}_{2\delta e^{\mu}}(p_{i}^{j})$ that contains $s_1$ and points of $\tilde{B}_{\delta}^{j}(p_{i})$.

      \item \textbf{Estimate of the Riccati equation solution for $g_{\overline{\rho}}$ in the generalized bubble.}
      
      Let $U^{s_{1}}_{\rho}(t)$ be a solution of the Riccati equation in $\gamma_{\rho}$ such that 
      $$U^{s_{1}}_{\rho} (s_{1})>e^{-\mu}\sqrt{\epsilon}\tanh(e^{-\mu}\frac{1}{2}\ln(3)).$$
      Let $U_{\overline{\rho}}^{s_{1}}$ be the Riccati solution in $\gamma_{\overline{\rho}}^{s_{1}}$ such that $U_{\overline{\rho}}^{s_{1}}(s_{1})=U^{s_{1}}_{\rho}(s_{1})$.
      By hypothesis, $U_{\overline{\rho}}^{s_{1}}(t)>e^{-2\mu}\epsilon\tanh^2(e^{-\mu}\frac{1}{3}\ln(3))$ for all $t\in(a,b)$. Furthermore, since $\tilde{\gamma}^{s_{1}}$ does not return to $\tilde{B}^{\overline{\rho}}_{\delta e^{\mu}}(p_{i}^{j})$ after $(a,b)$ and the distance between generalized bubbles is greater than $\Lambda$, by the behavior of the solution of the Riccati equation in the region of curvature less than $-\zeta$, it follows that $$U_{\overline{\rho}}^{s_{1}}(t)>e^{-2\mu}\epsilon\tanh^2(e^{-\mu}\frac{1}{3}\ln(3))$$
      for all $t\in(a,b+\Lambda)$. Fix $\lambda<\frac{\Lambda}{2}$.

      \item \textbf{Intersection between $\tilde{\gamma}_{\rho}$ and $\partial\tilde{B}_{\delta}^{j}(p_{i})$.}
      
      By continuous dependence of geodesics, there exists $r_{1}>0$ such that for any $\overline{\rho}\in[0,1]$ and $r_{1}\ge r'>0$, if $\gamma_{\overline{\rho}}$ is a $g_{\overline{\rho}}-$geodesic and $\gamma_{\overline{\rho}+r'}$ is a $g_{\overline{\rho}+r}-$geodesic with $\gamma_{\overline{\rho}+r'}(x)=\gamma_{\overline{\rho}}(x) \, \text{ and } \, \gamma_{\overline{\rho}+r'}'(x)=\frac{\gamma_{\overline{\rho}}'(x)}{||\gamma_{\overline{\rho}}'(x)||_{\overline{\rho}+r'}}$, then
     $$d_{\overline{\rho }}(\tilde{\gamma}_{\rho_{\overline{\rho }+r}}(t),\tilde{\gamma}_{\overline{\rho }}(t))\le d_{1}(\tilde{\gamma}_{\overline{\rho }+r}(t),\tilde{\gamma}_{\overline{\rho }}(t))< \frac{\lambda}{3}$$
      for all $t\in[x-4\delta e^{\mu}-\Lambda,x+4\delta e^{\mu}+\Lambda]$. This interval was chosen to ensure that $\tilde{\gamma}_{\overline{\rho}}$ does not return to the neighborhood of $\tilde{B}_{\delta}^{j}(p_{i})$.
      Thus, if $\rho\in [\overline{\rho},\rho_{l}]$ with $\rho_{l}=\min\{\overline{\rho}+r_{1},1\}$, there exists $\overline{s_{2}}\in (s_{1},b+\lambda)$ such that $\tilde{\gamma}_{\rho}(\overline{s}_{2})\in \partial\tilde{B}^{\overline{\rho}}_{2\delta e^{\mu}+\lambda}(p_{i}^{j})$.
      Note that $d(\tilde{\gamma}^{s_1}_{\overline{\rho}}(s_{1}\pm(4\delta e^{\mu}+2\lambda)),\tilde{B}^{j}_{\delta}(p_{i}))>\lambda$. 
      
      This also holds for $t\not\in[s_{1}-4\delta e^{\mu}-2\lambda,s_{1}+4\delta e^{\mu}+2\lambda]$.
      Let $s_{2}$ be the largest $s\in (s_{1},\overline{s}_{2})$ such that $\tilde{\gamma}_{\rho}(s)\in\partial \tilde{B}_{\delta}^{j}(p_{i})$, so $\gamma_{\rho}(s_{2})\in \partial B_{\delta}(p_{i})$.

      \item \textbf{Estimate of the Riccati equation solution for $g_{\rho}$ in the generalized bubble.}
      
      By continuous dependence of geodesics and Riccati solutions, there exists $r_{2}>0$ such that for any $\overline{\rho}\in[0,1]$ and $r_{2}\ge r'>0$, if $\gamma_{\overline{\rho}}$ and $\gamma_{\overline{\rho}+r'}$ are geodesics as before, and $U_{\overline{\rho}},U_{\overline{\rho}+r'}$ are their respective Riccati solutions with $U_{\overline{\rho}}(x )=U_{\overline{\rho}+r'}(x)$, then
      $$|U_{\overline{\rho}}(y)-U_{\overline{\rho}+r'}(y)|<\frac{e^{-2\mu}\epsilon\tanh^2(e^{-\mu}\frac{1}{3}\ln(3))}{4}$$
      for $y\in [x-4\delta e^{\mu}-\Lambda,x+4\delta e^{\mu}+\Lambda]$, that is, greater than the length of any $g_{\overline{\rho}}$-geodesic segment in a generalized bubble. Hence, if $\rho_{1}=\min \{\overline{\rho}+r_{1},\overline{\rho}+r_{2},1\}$ and $\gamma_{\rho} (s_2)$ is the exit point of $\gamma_{\rho}$ from $B_{\delta}(p_{i})$, then 
      $$U_{\rho}^{s_1}(s_{2})>\frac{3e^{-2\mu}\epsilon\tanh^2(e^{-\mu}\frac{1}{3}\ln(3))}{4}.$$

      \item \textbf{Estimation of the solution of the Riccati equation for $g_{\rho}$ between generalized bubbles.}

If $\tilde{\gamma}_{\rho}$ does not intersect again the lifting of a generalized bubble, we are done, since the geodesic $\gamma_{\rho}$ remains in the region of negative curvature. Then, suppose that $\tilde{\gamma}_{\rho}$ intersects once more the lifting of a generalized bubble $\tilde{B}^t_{\delta}(p_{l})$. 

By Proposition \ref{proposição nova}, $\tilde{\gamma}_{\rho}$ can only intersect $\tilde{B}_{\delta}^{j}(p_{i})$ again if it first intersects or encircles another lifting of a generalized bubble. Moreover, since the distance between the liftings of the generalized bubbles does not decrease as the parameter $\rho$ increases, and since
$$
U_{\rho}^{s_1}(s_{2})>\frac{3e^{-2\mu}\epsilon\tanh^2\!\left(e^{-\mu}\tfrac{1}{3}\ln(3)\right)}{4}>0,
$$
by Lemma \ref{lema de quanto cresce a solução da equação de Riccati a partir da escolha de Lambda} we have
$$
U_{\rho}^{s_1}> e^{-\mu} \sqrt{\epsilon} \,\tanh\!\left(\tfrac{1}{2} e^{-\mu} \ln(3)\right) 
> e^{-\mu} \sqrt{\epsilon} \,\tanh\!\left(\tfrac{1}{3} e^{-\mu} \ln(3)\right)
$$
    when re-entering a lifting of a generalized bubble. This allows us to restart the process from Step~1.

      \item \textbf{$(M,g_{\rho})$ is free of focal points.}
      
      Inductively, $U^{s_{1}}_{\rho}(t)>0$ for $t>s_{1}$. The choice of $s_{1}$ is arbitrary. If we choose a starting point outside a generalized bubble, the lower bound condition still holds when the geodesic enters it.
      Then by Lemma \ref{lema da convergência da solução da equação de riccati não negativa}, as in Theorem \ref{Teorema 1 introdução},there is a solution of the non-negative Riccati equation defined for all time and a solution of the non-positive Riccati equation. Furthermore, the modulus of these solutions is greater than or equal to $e^{-2\mu}\epsilon\tanh^2(e^{-\mu}\frac{1}{3}\ln(3))$.
      In particular, by Lemma \ref{lema instável não negativo e estável não positivo}, $\gamma_{\rho}$ has no focal points if $\rho\in[\overline{\rho},\rho_{1}]$ with
      $$\rho_{1}=\min \{\overline{\rho}+r,1\}$$
      where $r=\min\{r_1,r_2,r_3\}$.
      Since this holds for all $\gamma_{\rho}$, $(M,g_{\rho})$ is free of focal points and since $r>0$ is chosen uniformly, the lemma is proven for all $\overline{\rho}\in [0,1)$.
  \end{enumerate}
\end{proof}

\begin{figure}[H]
    \centering
    \includegraphics[scale=0.33]{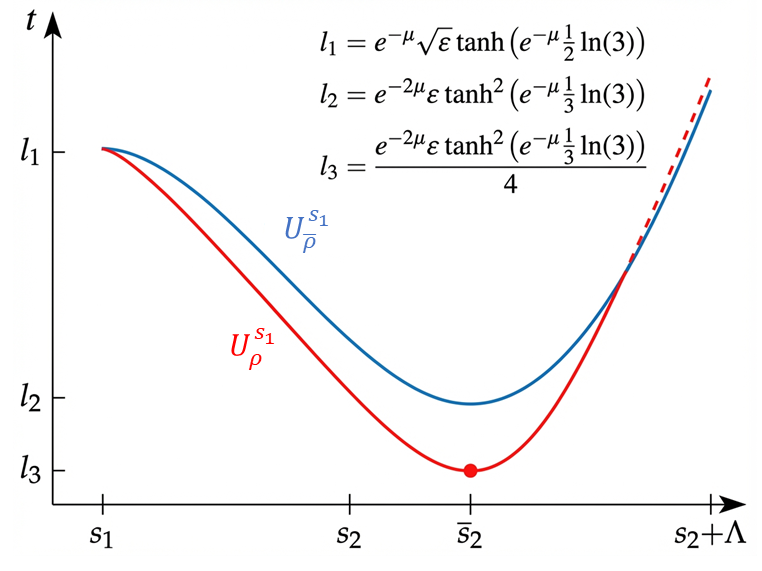}
    \caption{Illustration of the initial configuration of the geodesics and the regions involved in the sequence of arguments used in the proof of Proposition~\ref{proposição nova}.}
   
    \end{figure}

Note that $r>0$ is chosen uniformly, i.e., it does not depend on the choice of $\overline{\rho}\in[0,1]$. Moreover, in the proof we use both Proposition \ref{proposição nova}, which guarantees that the geodesic remains sufficiently long in the region of negative curvature, and Lemma \ref{lema não tem pontos focais na pertubação}, which controls the amount of time the geodesic is in the generalized bubble.

\medskip

Now, with the last lemma, we can finally prove the results about the non-existence of focal points and, subsequently, the result of the metric path generated by the deformation being Anosov metrics.

Let us now conclude with the proof of Proposition~\ref{prop deformação sem pontos locais}.\\

\begin{proof}
We prove Proposition~\ref{prop deformação sem pontos locais}.

Consider the deformation constructed in Proposition \ref{proposição da deformação}.
By hypotheses on the deformation, the curvature is smaller than $-\epsilon$ (and thus smaller than $-\zeta_{1}=-e^{-2\mu }\epsilon$) in $(\bigcup_{i=1}^{k}B_{\delta}(p_{i}))^{c}$.

As established in Lemma \ref{lema de quanto cresce a solução da equação de Riccati a partir da escolha de Lambda}, if $\gamma$ is a $g-$geodesic intersecting a generalized bubble in $(a,b)$, then by hypothesis, any Riccati solution $U$ with $U(b)>0$ satisfies 
$$U(b+\Lambda)>e^{-\mu}\sqrt{\epsilon}\,\tanh\!\Bigl(\tfrac12\,e^{-\mu}\ln(3)\Bigr).$$

Let $\gamma$ be a geodesic and $\{(a_{i},b_{i})\}_{i\in I}$ be the parameterized intervals where $\gamma$ intersects the generalized bubbles.
Using arguments from Lemma \ref{lema sai positivo da bolha e chega grande na próxima} and Theorem \ref{Teorema 1 introdução}, in a generalized bubble, if there exists $c\in(a_{2 },b_{2})$ such that $U(c)=e^{-\mu}\sqrt{\epsilon}\,\tanh\!\Bigl(\tfrac12\,e^{-\mu}\ln(3)\Bigr)$, then for $t \in [c, b_2]$,
\begin{align*}
    U(t) &= U(c) + \int_{c}^{t} U'(s) \, ds \\
    &= U(c) - \int_{c}^{t} (U^2(s) + K(s)) \, ds \\
    &\ge U(c) - \int_{c}^{t} (\max_{[c,b_{2}]} U(s))^2 \, ds - \int_{c}^{t} (\max_{[c,b_{2}]} K(s)) \, ds \\
    &\ge e^{-\mu}\sqrt{\epsilon}\,\tanh\!\Bigl(\tfrac12\,e^{-\mu}\ln(3)\Bigr) - (e^{-2\mu}\epsilon\tanh^2(e^{-\mu}\tfrac{1}{2}\ln(3)) + K^{+}) (t - c).
\end{align*}
Therefore, if
$$K^{+} <\frac{\sqrt{\epsilon}}{2\delta}\left[
  \tanh\!\left(e^{-\mu}\tfrac{1}{3}\ln3\right)
  -\sqrt{\epsilon}\,e^{-\mu}\,\tanh^{2}\!\left(e^{-\mu}\tfrac{1}{3}\ln3\right)\,(2\delta+1)
\right],$$
we have $U(t)>e^{-2\mu}\epsilon\tanh^2(e^{-\mu}\frac{1}{3}\ln(3))$ in the generalized bubble.

Given the hypothesis that
$$
K^{+}
<\frac{\sqrt{\epsilon}}{4e^{2\mu}\,\delta}\left[
  \tanh\!\left(e^{-\mu}\tfrac{1}{3}\ln3\right)
  -\sqrt{\epsilon}\,e^{-\mu}\,\tanh^{2}\!\left(e^{-\mu}\tfrac{1}{3}\ln3\right)\,(4e^{\mu}\delta+1)
\right],
$$
and utilizing Lemma \ref{lema de quanto cresce a solução da equação de Riccati a partir da escolha de Lambda} and Lemma \ref{lema sai da bolha e não volta mais depois de um tempo}, the Riccati solutions satisfy the conditions of Lemma \ref{lema não tem pontos focais na pertubação}. Thus, $(M,g_{\rho})$ has no focal points for $\rho\in[0,r]$.

Now, suppose $(M,g_\rho)$ has no focal points for $\rho\in[0,\overline{\rho}]$. This implies no focal points exist in $\bigcup_{i=1}^{k} B_{\delta}(p_{i})$.
Furthermore, if the distance between generalized bubbles in the $g$-metric is greater than $\Lambda$, this distance does not decrease under deformation (by Lemma \ref{lema distância entre as bolhas aumenta}) and thus remains greater than $\Lambda$ in the $g_{\overline{\rho}}$-metric. Consequently, if a Riccati solution exits a generalized bubble positively, it remains greater than $e^{-\mu}\sqrt{\epsilon}\,\tanh\! \Bigl(\tfrac12\,e^{-\mu}\ln(3)\Bigr)$ upon entering another generalized bubble, as the curvature remains less than $-\zeta_{1}$ in $(\bigcup_{i=1}^{k}B_{\delta}(p_{i}))^{c}$, by Lemma \ref{lema de quanto cresce a solução da equação de Riccati a partir da escolha de Lambda}. We must be careful about the case where the geodesic returns to the same generalized bubble, or rather the case where the lifted geodesic returns to the same lifting as the generalized bubble.

Specifically, by Lemma \ref{lema não tem pontos focais na pertubação}, it suffices that for any $\gamma_{\overline{\rho}}$ with $\gamma_{\overline{\rho}}(0)\in B_{\delta}(p_{i})$, if $\tilde{\gamma}_{\overline{\rho}}$ is its lifting such that $\tilde{\gamma}_{\overline{\rho}}(0)\in \tilde{B}^{j}_{\delta}(p_{i})$, and if $(a,b)$ is the connected component of $\tilde{\gamma}_{\overline{\rho}}\cap \tilde{B}^{\rho}_{2\delta e^{\mu}}(p^{j}_{i})$ containing $\tilde{\gamma}_{\overline{\rho}}(0)$, then if $U$ is a Riccati solution such that 
$$U(c)>e^{-\mu}\sqrt{\epsilon}\tanh(e^{-\mu}\frac{1}{3}\ln(3)) \,  \text{ for some } \, c\in(a,b),$$ 
then 
$$U(t)>e^{-2\mu}\epsilon\tanh^2(e^{-\mu}\frac{1}{3}\ln(3)) \, \text{ for all } \,  t\in(a,b).$$ 
This means we need to control the Riccati solution when the lifted geodesic exits the $e^{\mu}\delta$-neighborhood of the generalized bubble in the universal cover.

Furthermore, given that the length of the generalized bubbles in the deformed metric is smaller than $2e^{\overline{\rho}\mu}\delta$, $\Lambda=\frac{1}{\sqrt{\epsilon}}\artanh\left(1-\frac{\epsilon}{2}\right)$ as in Lemma \ref{lema de quanto cresce a solução da equação de Riccati a partir da escolha de Lambda}, and the region of positive curvature in $g_{\overline{\rho}}$ is contained in the generalized bubbles, if we denote $K^{+}_{\overline{\rho}}$ as the maximum curvature in $(M,g_{\overline{\rho}})$, an analogous calculation shows that if
$$
K^{+}_{\overline{\rho}}
<\frac{\sqrt{\epsilon}}{4e^{2\mu}\,\delta}\left[
  \tanh\!\left(e^{-\mu}\tfrac{1}{3}\ln3\right)
  -\sqrt{\epsilon}\,e^{-\mu}\,\tanh^{2}\!\left(e^{-\mu}\tfrac{1}{3}\ln3\right)\,(4e^{\mu}\delta+1)
\right],
$$
then $(M,g_{\rho})$ has no focal points for $\rho\in[0,\overline{\rho}+r]$, and the argument repeats inductively. This means $(M,g_{\rho})$ is free of focal points and satisfies the conditions of Lemma \ref{lema não tem pontos focais na pertubação}.

Thus, for $(M,g_{\rho})$ to have no focal points for $\rho\in [0,1]$, it is sufficient that for all $\rho\in[0,1]$, we have
$$
K^{+}_{\rho}
<\frac{\sqrt{\epsilon}}{4e^{2\mu}\,\delta}\left[
  \tanh\!\left(e^{-\mu}\tfrac{1}{3}\ln3\right)
  -\sqrt{\epsilon}\,e^{-\mu}\,\tanh^{2}\!\left(e^{-\mu}\tfrac{1}{3}\ln3\right)\,(4e^{\mu}\delta+1)
\right].
$$
Therefore, since $K_{\rho}^{+}<K^{+}$, if
$$
K^{+}
<\frac{\sqrt{\epsilon}}{4e^{2\mu}\,\delta}\left[
  \tanh\!\left(e^{-\mu}\tfrac{1}{3}\ln3\right)
  -\sqrt{\epsilon}\,e^{-\mu}\,\tanh^{2}\!\left(e^{-\mu}\tfrac{1}{3}\ln3\right)\,(4e^{\mu}\delta+1)
\right],
$$
$(M,g_{\rho})$ has no focal points for $\rho\in [0,1]$.
\end{proof}

Let us now conclude with the proof of Theorem~\ref{Teorema 3 introdução}.\\

\begin{proof}
We prove Theorem~\ref{Teorema 3 introdução}.

Let $w$ be as in Proposition~\ref{prop deformação sem pontos locais}.

By construction, the surface $(M,g_{\rho})$ has no focal points for every $\rho \in [0,1]$, as a consequence of the curvature bound
$$
K^{+}
<\frac{\sqrt{\epsilon}}{4e^{2\mu}\,\delta}\left[
  \tanh\left(e^{-\mu}\tfrac{1}{3}\ln 3\right)
  -\sqrt{\epsilon}\,e^{-\mu}\,\tanh^{2}\left(e^{-\mu}\tfrac{1}{3}\ln 3\right)\,(4e^{\mu}\delta+1)
\right].
$$

Moreover, by Theorem~\ref{Teorema 2 introdução}, the final metric $g_1$ has strictly negative curvature.

We now verify hyperbolicity. Although the argument follows that of Theorem~\ref{Teorema 1 introdução}, we present it again for completeness.

Since $(M,g_{\rho})$ has no conjugate points, stable and unstable Jacobi fields exist along every geodesic. According to \cite{eberlein1973geodesic}, it suffices to show that there are no nontrivial Jacobi fields which are simultaneously stable and unstable. Such a field could only exist on a corresponding geodesic that passes through a region of nonnegative curvature.

However, by our construction, along any $g_{\rho}$-geodesic, the unstable solution of the Riccati equation $U_{\rho}^{u}$ is positive, while the stable solution of the Riccati equation $U_{\rho}^{s}$ is negative. Therefore, the associated Jacobi fields cannot coincide, and hence no Jacobi field is simultaneously stable and unstable. This establishes hyperbolicity and concludes the proof.
\end{proof}

\bibliography{main} 

\begin{thebibliography}{10}

\bibitem{anosov1969geodesic}
D.~V. Anosov.
\newblock Geodesic flows on closed riemann manifolds of negative curvature.
\newblock {\em Proceedings of the Steklov Institute of Mathematics}, 90:3--210,
  1969.
\newblock Translated from the Russian by S. Feder.

\bibitem{aubin2012nonlinear}
Thierry Aubin.
\newblock {\em Nonlinear analysis on manifolds. Monge-Ampère equations},
  volume 252.
\newblock Springer-Verlag, 1982.

\bibitem{brin1974partially}
Michael~I Brin and Ja~B Pesin.
\newblock Partially hyperbolic dynamical systems.
\newblock {\em Mathematics of the USSR-Izvestiya}, 8(1):177, 1974.

\bibitem{burns1992flat}
Keith Burns.
\newblock The flat strip theorem fails for surfaces with no conjugate points.
\newblock {\em Proceedings of the American Mathematical Society}, pages
  199--206, 1992.

\bibitem{chang2005conformal}
Sun-Yung Chang.
\newblock Conformal invariants and partial differential equations.
\newblock {\em Bulletin of the American Mathematical Society}, 42(3):365--393,
  2005.

\bibitem{donnay2003anosov}
Victor~J. Donnay and Charles~C. Pugh.
\newblock Anosov geodesic flows for embedded surfaces.
\newblock {\em Ast{\'e}risque}, 287:61--69, 2003.

\bibitem{eberlein1973geodesic}
Patrick Eberlein.
\newblock When is a geodesic flow of anosov type? {I}.
\newblock {\em Journal of Differential Geometry}, 8(3):437--463, 1973.

\bibitem{eberlein1973geodesic2}
Patrick Eberlein.
\newblock When is a geodesic flow of anosov type? {II}.
\newblock {\em Journal of Differential Geometry}, 8(4):565--577, 1973.

\bibitem{eberlein1973visibility}
Patrick Eberlein and Barrett O’Neill.
\newblock Visibility manifolds.
\newblock {\em Pacific Journal of Mathematics}, 46(1):45--109, 1973.

\bibitem{green1958theorem}
Leon~W Green.
\newblock A theorem of {E. Hopf}.
\newblock {\em Michigan Mathematical Journal}, 5(1):31--34, 1958.

\bibitem{gulliver1975variety}
Robert Gulliver.
\newblock On the variety of manifolds without conjugate points.
\newblock {\em Transactions of the American Mathematical Society},
  210:185--201, 1975.

\bibitem{hedlund1939dynamics}
Gustav~A. Hedlund.
\newblock The dynamics of geodesic flows.
\newblock {\em Bulletin of the American Mathematical Society}, 45(4):241--260,
  1939.

\bibitem{hopf1939ergodentheorie}
Eberhard Hopf.
\newblock Ergodentheorie.
\newblock {\em Ergebnisse der Mathematik und ihrer Grenzgebiete}, 5:1--208,
  1939.

\bibitem{hopf1948closed}
Eberhard Hopf.
\newblock Closed surfaces without conjugate points.
\newblock {\em Proceedings of the National Academy of Sciences}, 34(2):47--51,
  1948.

\bibitem{jane2014boundary}
Dan Jane and Rafael~O Ruggiero.
\newblock Boundary of anosov dynamics and evolution equations for surfaces.
\newblock {\em Mathematische Nachrichten}, 287(17-18):2002--2020, 2014.

\bibitem{klingenberg1974riemannian}
Wilhelm Klingenberg.
\newblock Riemannian manifolds with geodesic flow of anosov type.
\newblock {\em Annals of Mathematics}, 99(1):1--13, 1974.

\bibitem{morse1924fundamental}
Harold~Marston Morse.
\newblock A fundamental class of geodesics on any closed surface of genus
  greater than one.
\newblock {\em Transactions of the American Mathematical Society},
  26(1):25--60, 1924.

\bibitem{pesin1977geodesic}
Ja~B Pesin.
\newblock Geodesic flows on closed riemannian manifolds without focal points.
\newblock {\em Mathematics of the USSR-Izvestiya}, 11(6):1195, 1977.

\bibitem{pesin1977characteristic}
Ya~B Pesin.
\newblock Characteristic lyapunov exponents and smooth ergodic theory.
\newblock {\em Russian Mathematical Surveys}, 32(4):55, 1977.

\bibitem{ruggiero1991creation}
Rafael~Oswaldo Ruggiero.
\newblock On the creation of conjugate points.
\newblock {\em Mathematische Zeitschrift}, 208:41--55, 1991.

\end{thebibliography}

\end{document}